\documentclass[12pt,fleqn,twoside]{article}

\newcommand{\versiondate}{{\footnotesize version October 20, 2025}} 
\usepackage{paralist}
\usepackage{a4}
\usepackage{amsthm}
\usepackage{enumitem}

\usepackage{ifthen}
\usepackage{amsmath}
\usepackage{amssymb}
\usepackage{graphicx} \usepackage{rotating}
\usepackage{color}
\usepackage{lscape}
\usepackage{mathtools} 
\usepackage{stmaryrd} 

\usepackage{textcomp} 

\swapnumbers
\theoremstyle{plain}
\newtheorem{theorem}{Theorem}[section]
\newtheorem{lemma}[theorem]{Lemma}
\newtheorem{proposition}[theorem]{Proposition}
\newtheorem{corollary}[theorem]{Corollary}

\theoremstyle{definition}
\newtheorem{definition}[theorem]{Definition}

\newtheorem{example}[theorem]{Example}

\newtheorem{remark}[theorem]{Remark}

\newtheorem{fact}[theorem]{Fact}

\newtheorem{conjecture}[theorem]{{\sc Conjecture}}

\numberwithin{equation}{theorem} 

\newcommand{\New}[1]{\emph{#1}}

\newcommand{\graph}[1]{#1^{\bullet}} 

\newcommand{\Orb}[1][]{\ifthenelse{\equal{#1}{}}{\textrm{Orb}}{\ensuremath{#1}\mbox{-}\textrm{Orb}}}

\newcommand{\um}{\underline{m}}

\DeclareMathOperator{\Sym}{Sym}
\DeclareMathOperator{\End}{End}
\DeclareMathOperator{\Con}{Con}
\DeclareMathOperator{\Pol}{Pol}
\DeclareMathOperator{\Inv}{Inv}
\newcommand{\Inva}[1][]{\Inv^{(#1)}} 
\DeclareMathOperator{\Op}{Op}
\newcommand{\Opa}[1][]{\Op^{(#1)}}
\DeclareMathOperator{\Rel}{Rel}
\newcommand{\Rela}[1][]{\Rel^{(#1)}}

\DeclareMathOperator{\Eq}{Eq} 
\DeclareMathOperator{\Equ}{Eq}

\DeclareMathOperator{\Quord}{Quord} 
\DeclareMathOperator{\gQuord}{gQuord} 
\newcommand{\gQuorda}[1][]{\gQuord^{(#1)}} 

\DeclareMathOperator{\gPord}{gPord} 
\newcommand{\gPorda}[1][]{\gPord^{(#1)}} 
\DeclareMathOperator{\wgPord}{wgPord} 
\newcommand{\wgPorda}[1][]{\wgPord^{(#1)}} 

\DeclareMathOperator{\gEq}{gEq} 
\newcommand{\gEqa}[1][]{\gEq^{(#1)}} 
\DeclareMathOperator{\gCon}{gCon} 
\newcommand{\gCona}[1][]{\gCon^{(#1)}} 

\DeclareMathOperator{\pr}{pr}

\newcommand{\preserves}{\mbox{ $\triangleright$ }}

\newcommand{\id}{\mathsf{id}}

\newcommand{\cK}{\mathcal{K}}
\newcommand{\cL}{\mathcal{L}}
\newcommand{\Mo}{\mathcal{M}}

 \newcommand{\bx}{\mathbf{x}}

\newcommand{\N}{\mathbb{N}}

\newcommand{\Sg}[2][]{\ensuremath{\langle #2 \rangle_{#1}}}

\newcommand{\tra}{\mathsf{tra}}

\newcommand{\trl}[1]{\mathsf{trl}(#1)} 

\newcommand{\tos}{\mathsf{tos}}  
\newcommand{\abs}{\mathsf{abs}}  

\newcommand{\refAU}[1]{\cite[#1]{JakPR2024}}

\let\rho=\varrho
\let\epsilon=\varepsilon
\let\phi=\varphi
\let\kappa=\varkappa

\let \restrictionORIGINAL=\restriction
\renewcommand{\restriction}{\mathclose\restrictionORIGINAL}

\newtheorem*{remarknote}{Remark}

\author{Danica Jakub{\'\i}kov\'a-Studenovsk\'a \footnote{supported by
     Slovak VEGA grant 1/0152/22}
  \\ Institute of
  Mathematics\\ P.J. \v{S}af\'arik University\\ Ko\v sice
  (Slovakia)
\and Reinhard P\"oschel%
 \\Institute of Algebra\\TU Dresden (Germany)
\and S\'andor Radeleczki%
\\Institute of Mathematics\\ University of
    Miskolc (Hungary)}
\date{\versiondate}

\title{
Generalized quasiorders:\\ constructions and characterizations}

\begin{document}

\maketitle


\begin{abstract}

Quasiorders $\rho\subseteq A^{2}$ have the property that  
an operation $f:A^{n}\to A$ preserves $\rho$ if and only if
  each (unary) translation obtained from $f$ is an endomorphism of
  $\rho$. Generalized quasiorders $\rho\subseteq A^{m} $ are
  generalizations of (binary) quasiorders sharing the same property.
  We show how new generalized quasiorders can be obtained
  from given ones using well-known algebraic constructions. 
 Special generalized quasiorders, as generalized
  equivalences and (weak) generalized partial orders, are introduced,
  which extend the
  corresponding notions for binary relations. 
  It turns out that generalized equivalences can be characterized by
  usual equivalence 
  relations. Extending some known results of binary quasiorders, 
  it is shown that generalized quasiorders can be
  \textquotedblleft decomposed\textquotedblright\ uniquely into a
  (weak) generalized partial order and a generalized equivalence.
  Furthermore, generalized quasiorders of maximal clones determined
  by equivalence or partial order relations are investigated.
  If $F=\Pol\rho$ is a maximal clone and $\rho$ an equivalence
  relation or a lattice order, then \textsl{every} relation in $\Inv F$
  is a generalized quasiorder. Moreover, lattice orders are
  characterized by this property among all partial orders.
Finally we prove that each term operation of a rectangular
  algebra gives rise to a generalized partial order.
Some problems requiring further research are also highlighted.
\end{abstract}

\section*{Introduction}\label{sec:intro}

Equivalence relations $\rho$ have the
remarkable well-known property that an
$n$-ary operation $f$ preserves $\rho$ (i.e., $f$ is a polymorphism of
$\rho$) if and only if each translation, i.e., unary polynomial function obtained from
$f$ by substituting constants, preserves $\rho$ (i.e., is an
endomorphism of $\rho$). Checking the proof one
sees that symmetry is not necessary, thus the same property, called
$\Xi$ in \cite[2.2]{JakPR2024}, also holds for quasiorders,
 i.e., reflexive and transitive relations.

With this ``motivating property'' $\Xi$ arose the problem whether there
are further relations satisfying $\Xi$. The answer is yes and given in
our paper \cite{JakPR2024} where the so-called 
\New{generalized quasiorders} are introduced and investigated (on the
base of the Galois connection $\gQuord-\End$). These
relations satisfy the property $\Xi$ (and each relation with $\Xi$ can be
constructed from generalized quasi\-orders).

The present paper is, in some sense, a continuation of
\cite{JakPR2024} in order to get more information which properties do
have generalized quasiorders and where they actually may appear. Since
the notion is new, not much was known up to now.

All needed basic notions and notation are introduced in
\New{Section~\ref{sec:prelim}}.
In the next \New{Section~\ref{sec:constructions}} we show how new generalized
quasiorders can be constructed from given ones using constructions
which might be called ``classical'' in algebra.

In \New{Sections~\ref{sec:equiv} and \ref{sec:gPord}} we introduce special
generalized quasiorders, namely
\New{generalized equivalences} and \New{(weak) generalized partial orders}
which generalize the corresponding notions for binary relations. As
Proposition~\ref{A2c} will show, generalized equivalences can be
uniquely characterized by usual binary equivalence relations and allow
factor constructions (\New{factor relations} and \New{block factor relations},
see Definition \ref{A3}). Under certain conditions these factor
constructions preserve the property of being a generalized quasiorder
(Proposition~\ref{A4}). In Theorem~\ref{thm:characterization} we shall
see that generalized quasiorders can be ``decomposed'' uniquely into a
(weak) generalized partial order and a (generalized) equivalence
(cf.~Corollary~\ref{cor:characterization}). This extends the
corresponding result for binary quasiorders.

With \New{Section~\ref{sec:maxclones}} we start to investigate
concrete algebras $(A,F)$ and ask for invariant
relations in $\Inv F$ being generalized quasiorders. If $F=\Pol\rho$
is a maximal clone determined by an equivalence relation $\rho$ then
\emph{every} invariant relation in $\Inv F$ is a generalized
quasiorder (Theorem~\ref{thm:eqrel}). The same is true if $\rho$ is a
lattice order on $A$ (and lattice
orders are even characterized by this property, see Theorem~\ref{thm:latticeorder}). In both cases
each invariant relation can be obtained from $\rho$ by a
quantifier-free primitive positive formula. For those maximal clones
$F=\Pol\rho$ with $\rho$ being a partial (but not a lattice) order with
$0$ and $1$ there always exist invariant relations which are not
generalized quasiorders (Example~\ref{ex:1}), however  we conjecture that each generalized
quasiorder in $\Inv F$ can be obtained from $\rho$ by a
quantifier-free pp-formula (see
Conjecture~\ref{conj:partialorder}). In the Boolean case (i.e., for
$|A|=2$) all generalized quasiorders can be characterized: they are
just the invariant relations of the clone of monotone Boolean
functions (Theorem~\ref{thm:boolean}).

Finally, in \New{Section~\ref{sec:rectangular}} we look to rectangular
algebras. Here each term operation gives rise to a generalized
quasiorder (which in fact is a generalized partial order,
cf.~Theorem~\ref{thm:rectangular}). In the last
\New{Section~\ref{sec:further}} we discuss some problems and questions for
further research.

\section{Preliminaries}\label{sec:prelim}

We briefly recall or introduce necessary notions and notation. For
more details we refer to \cite{JakPR2024}.

Throughout the paper, $A$ is a finite nonempty (base)
set. $\N:=\{0,1,2,\dots\}$ ($\N_{+}=\{1,2,\dots\}$, resp.) denote the set of
natural numbers (positive, resp.). For $m\in \N_{+}$ let
$\um:=\{1,\dots,m\}$.

\noindent
{\bf\large (A) \hspace{1ex} Operations, relations and clones}
\nopagebreak

  Let $\Opa[n](A)$ and $\Rela[n](A)$ denote the set of all $n$-ary
  operations $f:A^{n}\to A$ and $n$-ary relations $\rho\subseteq
  A^{n}$, $n\in\N_{+}$, respectively. Further, let
  $\Op(A)=\bigcup_{n\in \N_{+}}\Opa[n](A)$ and  
   $\Rel(A)=\bigcup_{n\in \N_{+}}\Rela[n](A)$. 

The so-called \New{projections}
$e^{n}_{i}\in\Opa[n](A)$ are defined by
$e^{n}_{i}(x_{1},\dots,x_{n}):=x_{i}$ ($i\in\{1,\dots,n\}$,
$n\in\N_{+}$). The identity mapping is denoted by $\id_{A}$
($=e^{1}_{1}$). 

For $f\in\Opa[n](A)$ and $m$-ary operations
$g_{1},\dots,g_{n}\in\Opa[m](A)$, the 
\New{composition} $f[g_{1},\dots,g_{n}]$ is the $m$-ary operation given by
$f[g_{1},\dots,g_{n}](\bx):=f(g_{1}(\bx),\dots,g_{n}(\bx))$, $\bx\in A^{m}$.

A \New{clone} is a set $F\subseteq \Op(A)$ of operations which contains
all constants and is closed under composition.

An $m$-ary relation $\delta\in\Rel(A)$
     ($m\in\N_{+}$) is called 
     \New{diagonal relation} if there exists an equi\-valence relation
     $\epsilon$ on the set $\{1,\dots,m\}$ of indices such that 
     $\delta=\{(a_{1},\dots,a_{m})\in A^{m}\mid \forall
     i,j\in\{1,\dots,m\}: (i,j)\in\epsilon\implies
     a_{i}=a_{j}\}$. With $D_{A}$ we denote the set of all diagonal
     relations (of arbitrary finite arity).

Special subsets of $\Rela[2](A)$ are
$\Equ(A)$ and $\Quord(A)$,
i.e., all
\New{equivalence relations} (binary, reflexive, symmetric and transitive) and
\New{quasiorder relations} (binary, reflexive and transitive), respectively,
on the set $A$.

For $f\in\Opa[n](A)$ and $r_{1},\dots,r_{n}\in A^{m}$,
$r_{j}=(r_{j}(1), \dots, r_{j}(m))$,
($n,m\in\N_{+}$, $j\in\{1,\dots,n\}$),
let $f(r_{1},\dots,r_{n})$ denote the $m$-tuple obtained from
componentwise application of $f$, i.e., the $m$-tuple

\centerline{$(f(r_{1}(1),\dots,r_{n}(1)),\dots,f(r_{1}(m),\dots,r_{n}(m)))$.}

An operation $f\in\Opa[n](A)$
  \New{preserves} a relation $\rho\in\Rela[m](A)$ ($n,m\in\N_{+}$) if
for all $r_{1},\dots,r_{n}\in\rho$ we have
  $f(r_{1},\dots,r_{n})\in\rho$, notation $f\preserves\rho$.

 The
Galois connection induced by $\preserves$ gives rise to
several operators as follows. For $Q\subseteq\Rel(A)$ and
$F\subseteq\Op(A)$ let
\begin{align*}
  \Pol Q&:=\{f\in\Op(A)\mid \forall \rho\in Q: f\preserves\rho\}
          &\text{ (\New{polymorphisms}),}\\
   \End Q&:=\{f\in\Opa[1](A)\mid \forall \rho\in Q: f\preserves\rho\} 
 &\text{ (\New{endomorphisms}),}\\
  \Inv F&:=\{\rho\in\Rel(A)\mid \forall f\in F: f\preserves\rho\} 
&\text{ (\New{invariant relations}),}\\
\Con F&:=\Con(A,F):=\Inv F\cap\Eq(A)&\text{ (\New{congruence relations}),}\\
\Quord F&:=\Quord(A,F):=\Inv F\cap\Quord(A)&\text{ (\New{compatible quasiorders}).}
\end{align*}
The Galois closures for $\Pol-\Inv$ are known and can be
characterized as follows: $\Pol\Inv F=\Sg{F}$ (clone generated by
$F$), $\Inv \Pol Q=[Q]_{(\exists,\land,=)}$ (relational clone generated
by $Q$, equivalently characterizable as closure with respect to
primitive positive formulas (pp-formulas), i.e., formulas containing
variable and relational symbols and only $\exists,\land,=$).
We refer to, e.g.,
\cite[1.2.1, 1.2.3, 2.1.3(i)]{PoeK79}, \cite{BodKKR69a},
\cite{Poe04a}, \cite{KerPS2014}.

We also need the closure $[Q]_{(\land,=)}$, i.e., the closure of $Q$
under constructions with quantifier-free pp-formulas (demonstrated
for a binary relation $\rho$ and $Q=\{\rho\}$ in
Lemma~\ref{lem:conjunctions}). 

\noindent
{\bf\large (B) \hspace{1ex} Generalized quasiorders}
\nopagebreak

As explained in the introduction, generalized quasiorders are in the
focus of our interest. We recall from \cite{JakPR2024}:

An $m$-ary relation $\rho\subseteq A^{m}$ ($m\in\N_{+}$) is called  \New{reflexive}
if $(a,\dots,a)\in\rho$ for all $a\in A$, and it is called
\New{(generalized) transitive} if for every $m\times m$-matrix
$(a_{ij})\in A^{m\times m}$ we have:
if every row and every column belongs to $\rho$ -- for this property we
write $\rho\models (a_{ij})$ -- then also the diagonal
$(a_{11},\dots,a_{mm})$ belongs to $\rho$, cf.~Figure~\ref{Fig1}.

\begin{figure}[h]
\includegraphics[scale=0.9]{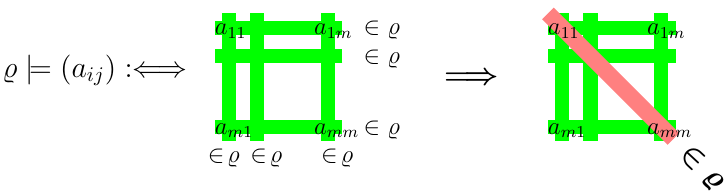}
\caption{Transitivity for an $m$-ary relation $\rho$}\label{Fig1}
\end{figure}

A reflexive and transitive relation is called \New{generalized
quasiorder}. The set of all generalized
quasiorders on the base set $A$ shall be denoted by $\gQuord(A)$, and
$\gQuorda[m](A):=\Rela[m](A)\cap\gQuord(A)$ will denote the $m$-ary
generalized quasiorders.

For $\rho\in \Rela[m](A)$ define 
\begin{align*}
  \partial(\rho)&:=\{(a_{11},\dots,a_{mm})\in A^{m}\mid \exists
(a_{ij})\in A^{m\times m}: \rho\models (a_{ij})\},\\
\rho^{(0)}&:=\rho,\quad
  \rho^{(n+1)}:=\rho^{(n)}\cup\partial(\rho^{(n)})\text{ for } n\in\N.
\end{align*}
 Then 
  $\rho^{\tra}=\bigcup_{n\in\N}\rho^{(n)}$ is the transitive closure
  of $\rho$, i.e., the least transitive relation containing $\rho$ (\cite[3.6]{JakPR2024}). If
  $\rho$ is reflexive then so it does $\rho^{\tra}$. Note that
  $\partial(\rho)$ is expressible via a pp-formula, i.e.,
  $\partial(\rho)\in [\rho]_{(\exists,\land,=)}$.

For $m=2$ we have 
$\partial(\rho)=\{(a,b)\mid \exists 
\left(\begin{smallmatrix}
  a&c\\d&b
\end{smallmatrix}\right): (a,c),(d,b),(a,d),(c,b)\in\rho\}
$, thus
$\partial(\rho)=\rho\circ\rho=\{(a,b)\mid \exists c\in A:
    (a,c)\in\rho\land(c,b)\in\rho\}$
    is the the usual relational product.

\section{Constructions with generalized
  quasiorders}\label{sec:constructions}

In this section we describe several ``classical'' constructions (relational constructions with $(\land,=)$-formulas,
substructures, products, homomorphic images) which produce new
generalized quasiorders from given ones. Further constructions (factor
relations) are considered later in Proposition~\ref{A4}.
  
In preparation of Proposition \ref{C2} we need the following lemma.

\begin{lemma}\label{C1}
  Let $\rho,\sigma\in\gQuorda[m](A)$ and let $\pi$ be a permutation of
  $\{1,\dots,m\}$. Then each of the following relations $\rho^{\pi}$,
  $\nabla\rho$, $\rho\land\sigma$ and $\Delta\rho$ is a
  generalized quasiorder:
  \begin{enumerate}[label=\textup{(\arabic*)}]
  \item\label{C1-1} \text{\rm permutation of coordinates:}\\
   $\rho^{\pi}:=\{(a_{\pi 1},\dots,a_{\pi m})\in A^{m}\mid 
                     (a_{1},\dots,a_{m})\in\rho\}$,

  \item\label{C1-2} \text{\rm adding a fictitious coordinate:}\\
    $\nabla\rho:=\{(a_{1},\dots,a_{m+1})\in A^{m+1}\mid
    (a_{1},\dots,a_{m})\in\rho\}$,
  \item\label{C1-3} \text{\rm intersection:}
  $\rho\land \sigma:=\rho\cap\sigma$,

  \item\label{C1-4} \text{\rm identification of coordinates:}\\
 $\Delta\rho:=\{(a_{1},\dots,a_{m-1})\in A^{m-1}\mid
    (a_{1},a_{1}\dots,a_{m})\in\rho\}$.

  \end{enumerate}
\end{lemma}

\begin{proof} The proof is straightforward using the
  definitions. E.g.~\ref{C1-2}:\\
  $\nabla\rho\models (b_{ij})_{i,j\in\underline{m+1}}\implies
   \rho\models
   (b_{ij})_{i,j\in\um}\stackrel{\text{transitivity}}{\implies}
   (b_{ii})_{i\in\um}\in\rho\implies (b_{ii})_{i\in\underline{m+1}}\in\nabla\rho$.
\end{proof}

\begin{proposition}[constructions with $(\land,=)$-formulas]\label{C2}
  Let $Q\subseteq\gQuord(A)$. Then each relation obtained from $Q$
  by a quantifier-free primitive positive formula is also a
  generalized quasiorder, i.e.,
  $[\gQuord(A)]_{(\land,=)}=\gQuord(A)$. 
Moreover, given an algebra $(A,F)$,
  $F\subseteq\Op(A)$, and $Q\subseteq\gQuord(A,F)$, then we have
  $[Q]_{(\land,=)}\subseteq\gQuord(A,F)$.

\end{proposition}

\begin{proof} Let $\phi(x_{1},\dots,x_{m})\in\Phi(\land,=)$, i.e., $\phi$
  is a conjunction of atomic formulas of the form
  $\phi':\equiv(x_{i_{1}},\dots,x_{i_{s}})\in\sigma$ 
  (with $s$-ary $\sigma\in Q$ and $i_{1},\dots,i_{s}\in\{1,\dots,m\}$).
  W.l.o.g. we add the
  diagonal $\Delta_{A}\in\gQuorda[2](A)$ to $Q$, therefore the atomic formulas $x_{i}=x_{j}$ are included: $x_{i}=x_{j}\iff
  (x_{i},x_{j})\in\Delta_{A}$. Moreover, we can assume that the
  $i_{1},\dots,i_{s}$ are pairwise distinct (otherwise, if
  $i_{j}=i_{j'}$, the atomic 
  formula $\phi'$ can be replaced by
  $(x_{i_{1}},\dots,x_{i_{s}})\in\sigma\land x_{i_{j}}=x_{i_{j'}}$).
  For $\phi'$, we consider the $m$-ary relation $\tau_{\phi'}:=\{(a_{1},
  \dots,a_{m})\in A^{m}\mid (a_{i_{1}},\dots,a_{i_{s}})\in\sigma\}$ (all
  components except the selected $i_{1},\dots,i_{s}$ are
  fictitious). Clearly, the $m$-ary relation, say $\rho$,
 defined by $\phi$ is the
  intersection of all $\tau_{\phi'}$ ($\phi'$ beeing an atomic
  formula of $\phi$). 

Because $\gQuord(A)$ is closed under intersection (cf.~\ref{C1}\ref{C1-3}) and
adding fictitious components (cf.~\ref{C1}\ref{C1-2}), and because $\tau_{\phi'}$
can be obtained from $\sigma\in Q\subseteq \gQuord(A)$ by adding
fictitious components (and possibly permutation of coordinates, \ref{C1}\ref{C1-1}), we can
conclude that $\rho$ is also a generalized quasiorder.
\end{proof}

\begin{remarknote}
 It is known that the closure $[Q]_{(\land,=)}$
is equivalent to the closure under the following
  constructions:
intersection, adding fictive components, identification of components,
doubling of components, using the trivial unary relation $A$ (then all
diagonal relations $D_{A}$ are included), for more details concerning
these operations we refer to \cite[pp.~42, 43, 67]{PoeK79}.
\end{remarknote}


For $\rho_{1}\in\Rela[m](A_{1})$, $\rho_{2}\in\Rela[m](A_{2})$ and
$A:=A_{1}\times A_{2}$
we define the ``direct product'' as follows (cf.\ \refAU{4.6}):
\begin{align*}
    \rho_{1}\otimes\rho_{2}:=\{((a_{1},b_{1}),\dots,(a_{m},b_{m}))\in
                             A^{m}\mid 
(a_{1}&,\dots,a_{m})\in\rho_{1}\\&\land (b_{1},\dots,b_{m})\in\rho_{2}\}.
\end{align*}
The following proposition shows that this construction is compatible
with generalized quasiorders. This was stated and already proved as a part of
the proof of \refAU{Proposition~4.7(i)}.

\begin{proposition}[direct products]\label{dirprod}
   \begin{align*}
    \rho_{1}\otimes\rho_{2}\in\gQuord(A) \iff
    \rho_{1}\in\gQuord(A_{1}) \text{ and }\rho_{2}\in\gQuord(A_{2}).\qed
  \end{align*}
\end{proposition}


Now let us consider restrictions to subsets. Let $\emptyset\neq B\subseteq A$ and $\rho\restriction_{B}:=\rho\cap
B^{m}$ for $\rho\in\Rela[m](A)$. For a set $Q\subseteq\Rel(A)$ we put
$Q\restriction_{B}:=\{\rho\restriction_{B}\mid \rho\in Q\}$. Then we
have the following result (which is a special case of \refAU{4.8}, namely for $M=\{\id_{A}\}$).

\begin{proposition}[restriction to subsets]\label{restriction}
      $\gQuord(A)\restriction_{B}=\gQuord(B)$.\qed
\end{proposition}


Finally we consider ``homomorphisms''. Let $\lambda:A\to B$ be a surjective mapping. For $\rho\in\Rela[m](A)$
and $\sigma\in\Rela[m](B)$ we define
\begin{align*}
  \lambda(\rho)&:=\{(\lambda a_{1},\dots,\lambda a_{m})\in B^{m}\mid
                  (a_{1},\dots,a_{m})\in\rho\}\in\Rela[m](B),\\
  \lambda^{-1}(\sigma)&:=\{(a_{1},\dots,a_{m})\in A^{m}\mid
                 (\lambda a_{1},\dots,\lambda a_{m})\in\sigma\}\in\Rela[m](A).
\end{align*}
Then we have:

\begin{proposition}\label{homomorphism}
  If $\rho\in\gQuord(A)$ has the property
  $\lambda^{-1}(\lambda(\rho))=\rho$ then
  $\lambda(\rho)\in\gQuord(B)$. Conversely, if $\sigma\in\gQuord(B)$
  and $\rho:=\lambda^{-1}(\sigma)$,
  then $\rho\in\gQuord(A)$ and $\rho$
  has the property $\lambda^{-1}(\lambda(\rho))=\rho$.
\end{proposition}

\begin{proof} Let $\rho\in\gQuord(A)$ satisfy
  $\lambda^{-1}(\lambda(\rho))=\rho$.
Then $\lambda(\rho)$ is reflexive. In fact, let $b\in B$, then there
exist $a\in A$ with $b=\lambda(a)$ (since $\lambda$ is
surjective). Then $(a,\dots,a)\in\rho$ (by reflexivity of $\rho$)
implies $(b,\dots,b)\in\lambda(\rho)$. It remains to show transitivity
of $\lambda(\rho)$. Let $\lambda(\rho)\models (b_{ij})_{i,j\in\um}$.
By surjectivity of $\lambda$ we can choose $a_{ij}\in A$ such that
$\lambda(a_{ij})=b_{ij}$. Thus each row and column of
$(a_{ij})_{i,j\in\um}$ belongs to $\lambda^{-1}(\lambda(\rho))=\rho$,
consequently $\rho\models (a_{ij})_{i,j\in\um}$. By transitivity of
$\rho$ we get $(a_{11},\dots,a_{mm})\in\rho$, thus
$(b_{11},\dots,b_{mm})=(\lambda(a_{11}),\dots,\lambda(a_{mm}))\in\lambda(\rho)$
proving transitivity of $\lambda(\rho)$.

Conversely, let $\sigma\in\gQuord(B)$ and
$\rho:=\lambda^{-1}(\sigma)$. Then $\rho$ is reflexive since each
$(a,\dots,a)\in A^{m}$ belongs to $\lambda^{-1}(\sigma)$ (by
surjectivity of $\lambda$ and reflexivity of $\sigma$). $\rho$ is also
transitive: Let $\rho\models(a_{ij})_{i.j\in\um}$, then each row and
column belongs to $\rho=\lambda^{-1}(\sigma)$, thus the
$\lambda$-image of each row and column is in $\sigma$, i.e.,
$\sigma\models(\lambda(a_{ij}))_{ij\in\um}$. By transitivity of
$\sigma$ we have $(\lambda(a_{11}),\dots,\lambda(a_{mm}))\in\sigma$,
consequently $(a_{11},\dots,a_{mm})\in\lambda^{-1}(\sigma)=\rho$,
showing transitivity of $\rho$. It remains to mention that
$\lambda^{-1}(\lambda(\rho))=\lambda^{-1}(\lambda(\lambda^{-1}(\sigma)))=\lambda^{-1}(\sigma)=\rho$.
\end{proof}

The property $\lambda^{-1}(\lambda(\rho))=\rho$ which plays a crucial
role in \ref{homomorphism}, can be characterized by a property of
$\ker\lambda$ as follows.

\begin{definition}[exchange property]\label{def:exchange}
Let $\psi\in\Eq(A)$ and $\rho\in\Rela[m](A)$ 

Then we say that $\psi$ has the \New{exchange property with respect
  to $\rho$}
if $(a_{1},\dots,a_{m})\in\rho$ and $(a_{i},b_{i})\in\psi$ for
$i\in\{1,\dots,m\}$ implies $(b_{1},\dots,b_{m})\in\rho$,
equivalently, if 
\begin{align*}\label{exch1}\tag*{\ref{def:exchange}(1)}
  (a_{1},\dots,a_{m})\in\rho\implies [a_{1}]_{\psi}\times\ldots\times[a_{m}]_{\psi}\subseteq\rho
\end{align*}
for all $a_{1},\dots,a_{m}\in A$.
  
\end{definition}

\begin{proposition}\label{prop:lambdaexch} Let $\lambda:A\to B$ be a
  surjective mapping and $\rho\in\Rel(A)$.

Then $\lambda^{-1}(\lambda(\rho))=\rho$ if and only if $\ker\lambda$
has the exchange property with respect to $\rho$.
  
\end{proposition}

\begin{proof} Let $\rho$ be $m$-ary.

``$\Longrightarrow$'': Assume $\lambda^{-1}(\lambda(\rho))=\rho$ and
let $(a_{1},\dots,a_{m})\in\rho$, $(a_{i},b_{i})\in\ker\lambda$. Then
$(\lambda b_{1},\dots,\lambda b_{m})=
      (\lambda a_{1}, \dots,\lambda a_{m})\in\lambda(\rho)$,
thus $(b_{1},\dots,b_{m})\in\lambda^{-1}(\lambda(\rho))=\rho$, i.e.,
$\ker\lambda$ has the exchange property w.r.t.\ $\rho$.

``$\Longleftarrow$'': Assume that $\ker\lambda$
has the exchange property with respect to $\rho$. It is suffient to
prove $\lambda^{-1}(\lambda(\rho))\subseteq\rho$ (the other inclusion
is trivial).\\ Let $(b_{1},\dots,b_{m})\in\lambda^{-1}(\lambda(\rho)) $,
i.e., $(\lambda b_{1},\dots,\lambda b_{m})\in\lambda(\rho)$. Therefore
there must exist $(a_{1},\dots,a_{m})\in\rho$ with $(\lambda
b_{1},\dots,\lambda b_{m})=(\lambda a_{1},\dots,\lambda
a_{m})$, i.e., $(a_{i},b_{i})\in\ker\lambda$. Thus $(b_{1},\dots,b_{m})\in
[a_{1}]_{\ker\lambda}\times\ldots\times[a_{m}]_{\ker\lambda}\stackrel{\ref{exch1}}{\subseteq}\rho$,
  showing $\lambda^{-1}(\lambda(\rho))\subseteq\rho$.
\end{proof}

\section{Generalized equivalence relations}\label{sec:equiv}

Equivalence relations are symmetric quasiorders (reflexive,
transitive). We generalize the notion of symmetry to $m$-ary relations
in order to get an $m$-ary counterpart of ordinary equivalence relations.

\begin{definition}\label{def:gEq}
For a mapping $\alpha:\{1,\dots,m\}\to \{1,\dots,m\}$ (we can write
$\alpha\in\um^{\um}$)  and $\rho\subseteq A^{m}$
we define (cf.~\ref{C1}\ref{C1-1} for permutations)
\begin{align*}
  \rho^{\alpha}&:=\{(a_{\alpha 1},\dots,a_{\alpha m})\mid
                 (a_{1},\dots,a_{m}) \in\rho\}.
\end{align*}
A relation $\rho\subseteq A^{m}$ is called \New{totally symmetric} or 
\New{absolutely symmetric}, respectively, if
$\rho^{\alpha}\subseteq\rho$ (i.e., $(a_{1},\dots,a_{m})\in\rho\implies
(a_{\alpha 1},\dots,a_{\alpha m})\in\rho$) for all permutations $\alpha\in\Sym(\{1,\dots,m\})$ or
for all mappings $\alpha\in\um^{\um}$, respectively. Equivalently, an
absolutely symmetric relation $\rho$ can be defined by the property that
$\{a_{1},\dots,a_{m}\}^{m}\subseteq\rho$ for all
$(a_{1},\dots,a_{m})\in\rho$.

A relation $\rho\subseteq A^{m}$ is called \New{generalized
  equivalence relation} if it is reflexive, totally symmetric and
transitive, i.e., if it is a totally symmetric generalized
quasiorder. The set of 
all generalized equivalence relations is denoted by $\gEq(A)$
($\gEqa[m](A)$ for $m$-ary relations). In
Proposition~\ref{A2c} 
we shall see that a generalized equivalence relation is also
absolutely symmetric.

For a relation $\rho\in\Rela[m](A)$, the \New{totally
  symmetric part} $\tos(\rho)$ and the \New{absolutely
  symmetric part} $\abs(\rho)$, resp., of $\rho$ are defined 
as follows:
\begin{align*}
\label{gEq-1}
\tos(\rho)&:=\{(a_{1},\dots,a_{m})\mid
    \forall\alpha\in\Sym(\um): (a_{\alpha 1},\dots,a_{\alpha m})\in\rho\}
\tag*{\ref{def:gEq}(1)}\\
 &\phantom{:}=\bigcap\{\rho^{\alpha}\mid \alpha\in\Sym(\um)\},\\
\label{gEq-2}
  \abs(\rho)&:=\{(a_{1},\dots,a_{m})\mid\{a_{1},\dots,a_{m}\}^{m}\subseteq\rho\}
\tag*{\ref{def:gEq}(2)}\\
\label{gEq-3}
      &\phantom{:}=\{(a_{1},\dots,a_{m})\mid
    \forall\alpha\in\um^{\um}: (a_{\alpha 1},\dots,a_{\alpha m})\in\rho\}.
\tag*{\ref{def:gEq}(3)}
\end{align*}

For the second characterization of $\tos(\rho)$ note that 
$ (a_{1},\dots,a_{m})\in\rho^{\alpha}\iff 
   (a_{\alpha^{-1} 1},\dots,a_{\alpha^{-1} m})\in\rho$
 and $\alpha\in\Sym(\um)\iff\alpha^{-1}\in\Sym(\um)$.
The equivalence of (2) and (3) is a consequence
of the fact that for any elements $a_{1},\dots,a_{m}\in A$ we have
$\{a_{1},\dots,a_{m}\}^{m}=\{(a_{\alpha 1},\dots,a_{\alpha m})\mid
\alpha\in\um^{\um}\}$.

The \New{binary symmetric part} of $\rho$ is defined by
\begin{align*}\label{gEq-4}\tag*{\ref{def:gEq}(4)}
  \rho^{[2]}:=\{(a,b)\in A\mid \{a,b\}^{m}\subseteq\rho\}.
\end{align*}

\end{definition}

\begin{remark}\label{rem:gEq}
In general,
$\tos(\rho)$, $\abs(\rho)$ as well as $\rho^{[2]}$ can be empty. However, for reflexive $\rho$
the definition implies that $\tos(\rho)$ ($\abs(\rho)$, resp.) is a
reflexive and totally (absolutely, resp.)
symmetric relation.
Likewise $\rho^{[2]}$ is
reflexive and symmetric, i.e., a tolerance relation. 

Moreover we have
$\abs(\rho^{[2]})=\rho^{[2]}=(\abs(\rho))^{[2]}$ as can be checked easily.
Clearly, $\tos(\rho)\subseteq\rho$ and $\abs(\rho)\subseteq\rho$, and $\rho$ is totally (absolutely, resp.) symmetric if and only
if $\tos(\rho)=\rho$ ($\abs(\rho)=\rho$, resp.).

\begin{lemma}\label{lem:intersection1}
  Let $\rho_{i}\in\Rela[m](A)$, $i\in I$. Then
  \begin{enumerate}[label=\textup{(\alph*)}]

  \item\label{intersection1-a}
$  \bigcap_{i\in I}\tos(\rho_{i})=\tos(\bigcap_{i\in
      I}\rho_{i})$ and $\bigcap_{i\in I}\abs(\rho_{i})=\abs(\bigcap_{i\in
      I}\rho_{i})$,

  \item\label{intersection1-b} $\bigcap_{i\in I}\rho_{i}^{[2]}=(\bigcap_{i\in I}\rho_{i})^{[2]}$, 

  \item\label{intersection1-c} The construction $\rho\mapsto\rho^{[2]}$ is monotone with respect
    to inclusion, i.e., $\sigma\subseteq\rho\in\Rela[m](A)$ implies
    $\sigma^{[2]}\subseteq\rho^{[2]}$.
  \end{enumerate}
\end{lemma}

\begin{proof}
  \ref{intersection1-a} and   \ref{intersection1-c}  directly follow from the definitions.

\ref{intersection1-b}: We have 
$(a,b)\in(\bigcap_{i\in I}\rho_{i})^{[2]} $ $\iff$
$\{a,b\}^{m}\subseteq \bigcap_{i\in I}\rho_{i}$\\ $\iff$ 
$\forall i\in I: \{a,b\}^{m}\subseteq\rho_{i} $ $\iff$ 
$\forall i\in I: (a,b)\in\rho_{i}^{[2]} $ $\iff$ 
$(a,b)\in\bigcap_{i\in I}\rho_{i}^{[2]}$.
\end{proof}

\end{remark}

For generalized quasiorders $\rho$ we have some additional properties:

\begin{proposition}\label{A2a}
  Let $\rho\in\gQuord(A)$. 
  \begin{enumerate}[label=\textup{(\alph*)}]
  \item\label{A2a-a} $\abs(\rho)$ is a generalized
  equivalence relation (in particular
  $\abs(\rho)\in\gQuord(A)$), and it is the largest generalized
  equivalence relation contained in $\rho$. Moreover, $\tos(\rho)=\abs(\rho)$.
\item \label{A2a-b} $\rho^{[2]}$ is an equivalence relation.
  \end{enumerate}

\end{proposition}

\begin{proof}\ref{A2a-a}: Since $\um$ is finite, the condition $\forall \alpha\in
\um^{\um}:(a_{\alpha 1},\dots,a_{\alpha m})\in\rho$ is a finite
conjunction of atomic formulas $(a_{\alpha 1},\dots,a_{\alpha
  m})\in\rho$. Thus we have 
$\abs(\rho)\in [\rho]_{(\land,=)}$ by Definition~\ref{gEq-3}. Because
$\rho\in\gQuord(A)$ we get $\abs(\rho)\in\gQuord(A)$ from
Proposition~\ref{C2}. Since $\abs(\rho)$ is totally symmetric we have
$\abs(\rho)\in\gEq(A)$. Analogously we get $\tos(\rho)\in\gEq(A)$.

Further, $\abs(\rho)$ contains every generalized equivalence relation
$\theta$ contained in $\rho$. In fact, $\theta\subseteq\rho$ implies
$\abs(\theta)\subseteq\abs(\rho)$. In
Proposition~\ref{A2c}\ref{A2c-A}\ref{A2c-ii} we shall see that
$\abs(\theta)=\theta$, thus $\theta\subseteq\abs(\rho)$. Moreover,
for $\theta=\tos(\rho)\in\gEq(A)$, this implies $\tos(\rho)\subseteq\abs(\rho)$ what gives equality $\tos(\rho)=\abs(\rho)$
since the other inclusion $\tos(\rho)\supseteq\abs(\rho)$ is trivial
by the definitions \ref{gEq-1} and \ref{gEq-3}.  

\ref{A2a-b}: The defining condition
  $\{a,b\}^{m}\subseteq\rho$ (cf.~\ref{gEq-4}) can be
  expressed formally as a
  conjunction of atomic formulas of the form
  $(\dots,a,\dots,b,\dots)\in\rho$. Thus
  $\rho^{[2]}\in[\rho]_{(\land,=)}$ what implies that
  $\rho^{[2]}$ is a generalized quasiorder by
  Proposition~\ref{C2}. Binary generalized quasiorders are just
  quasiorders (reflexive and transitive). Symmetry directly follows
  from the definition, see also Remark \ref{rem:gEq}, i.e., we have
  $\rho^{[2]}\in\Eq(A)$. 
\end{proof}

\begin{lemma}\label{A2b}
  Let $\theta$ be an $m$-ary generalized equivalence relation ($m\geq
  2$) and $(a_{1},\dots,a_{m})\in\theta$. Then
  $\{a_{i},a_{j}\}^{m}\subseteq \theta$ for any $i,j\in\um$.
\end{lemma}

\begin{proof} Let
  $(a_{1},\dots,a_{m})\in\theta$. If $a_{i}=a_{j}$ then trivially
  $\{a_{i},a_{j}\}^{m}\subseteq\theta$ by reflexivity of
  $\theta$, so we may assume that there are at least two different
  elements among $a_{1},\dots,a_{m}$.

At first we show
  $(a_{i},\dots,a_{i},a_{j})\in\rho$ for all $i,j\in\um$. 
Let $(b_{1},\dots,b_{m})=(a_{\alpha 1},\dots,a_{\alpha m})$ for some
permutation $\alpha$ on $\{1,\dots,m\}$ such that $b_{1}\neq b_{m}$. Since $\theta$ is totally
symmetric we have $(b_{1},\dots,b_{m})\in\theta$.
Consider the matrix
\begin{align*}
M:= \begin{pmatrix}
   b_{1}&b_{2}&b_{3}&\dots&b_{m-1}&b_{m}\\
   b_{m-1}&b_{1}&b_{2}&\dots&b_{m-2}&b_{m}\\
   b_{m-2}&b_{m-1}&b_{1}&\dots&b_{m-3}&b_{m}\\
   \vdots&\vdots&\vdots&\dots&\vdots&\vdots\\
   b_{2}&b_{3}&b_{4}&\dots&b_{1}&b_{m}\\
   b_{m}&b_{m}&b_{m}&\dots&b_{m}&b_{m}
 \end{pmatrix}
 \text{ with diagonal $(b_{1},\dots,b_{1},b_{m})$.}
\end{align*}

Each of the first $n-1$ rows and columns is a permutation of the tuple
$(b_{1},\dots,b_{m})\in\theta$ and therefore also belongs to $\theta$
(since  $\theta$ is totally symmetric) and the last row and column
belong to $\theta$ by reflexivity. Thus
$\theta\models M$, hence $(b_{1},\dots,b_{1},b_{m})\in\theta$ by
transitivity of $\theta$.

For $a_{i}\neq a_{j}$  we
can apply this with any permutation satisfying $\alpha:1\mapsto i, m\mapsto j$ 
 what gives $(a_{i},\dots,a_{i},a_{j})\in \theta$ (for $a_{i}=a_{j}$
 this is trivially true). Permuting the components we conclude that
 each tuple $(a_{i},\dots,a_{j},\dots,a_{i})\in\{a_{i},a_{j}\}^{m}$
 (where $a_{j}$ appears only once) also belongs to $\theta$.

Now let $(c_{1},\dots,c_{m})$ be an arbitrary tuple in
$\{a_{i},a_{j}\}^{m}$. Consider the matrix $(a_{k\ell})_{k,\ell\in\um}$ given by
%
 $ a_{k\ell}:=
  \begin{cases}
    c_{k} &\text{if }k=\ell,\\
    a_{i}&\text{otherwise,}
  \end{cases}$
with diagonal $(c_{1},\dots,c_{m})$. Note that each row and column
contains at most one $a_{j}$ (namely if the diagonal element $c_{k}$
equals $a_{j}$), all other
entries are $a_{i}$, and thus all rows and columns belongs to $\theta$ (as shown above). Therefore $\theta\models(a_{k\ell})$ and by
transitivity of $\theta$ we have
$(c_{1},\dots,c_{m})\in\theta$. Consequently $\{a_{i},a_{j}\}^{m}\subseteq\theta$.
\end{proof}

\begin{theorem}\label{A2c} 

  \begin{enumerate}[label=\textup{(\Alph*)}]
\item\label{A2c-A}Let $\theta$ be an $m$-ary ($m\geq 2$) generalized
  equivalence relation. Then
  \begin{enumerate}[label=\textup{(\roman*)}]
\item \label{A2c-i} The binary symmetric part $\theta^{[2]}$ (cf.~{\rm\ref{gEq-4}})
completely determines $\theta$, namely
we have:
  \begin{align*}\label{A2c-1}\tag*{\ref{A2c}(1)}
   (a_{1},\dots,a_{m})\in\theta&\iff 
        \forall i,j\in\um: (a_{i},a_{j})\in\theta^{[2]}.
  \end{align*}
\item\label{A2c-ii} $\abs(\theta)=\theta$, in particular, $\theta$ is absolutely
symmetric.
  \end{enumerate}
\item\label{A2c-B} Let $\psi\in\Eq(A)$. Then $\psi^{\updownarrow m}$
  defined by
  \begin{align*}\label{A2c-2}\tag*{\ref{A2c}(2)}
    \psi^{\updownarrow m}:=\{(a_{1},\dots,a_{m})\in A^{m}\mid \forall i,j\in\um: (a_{i},a_{j})\in\psi\}
  \end{align*}
is a generalized equivalence relation. 
\item\label{A2c-C}
We have 
$(\psi^{\updownarrow m})^{[2]}=\psi$ and $(\theta^{[2]})^{\updownarrow m}=\theta$
for $\psi\in\Eq(A)$ and $\theta\in\gEqa[m](A)$. The mappings
$\psi\mapsto \psi^{\updownarrow m}$ and $\theta\mapsto\theta^{[2]}$ are
isomorphisms between the lattices $\Eq(A)$ and $\gEqa[m](A)$.

\end{enumerate}
\end{theorem}

\begin{proof}\ref{A2c-A}\ref{A2c-i}: 
To show \ref{A2c-1}, note that
$(a_{1},\dots,a_{m})\in\theta$ implies $
\{a_{i},a_{j}\}^{m}\subseteq\theta$ by Lemma~\ref{A2b}.
 To see the inverse implication,
let $\{a_{i},a_{j}\}^{m}\subseteq\theta$ for all $i,j\in\um$. Then we have
$\theta\models(a_{ij})_{i,j\in\um}$ for the matrix with
$a_{ii}:=a_{i}$ for $i\in\um$,
and $a_{ij}:=a_{1}$ otherwise (each row and column contains at most
two different entries and therefore belongs to $\theta$). By transitivity of $\theta$ we get for
the diagonal $(a_{1},\dots,a_{m})\in\theta$, what was to be shown.

\ref{A2c-ii}: We have $\abs(\theta)\subseteq \theta$ by definition. We
have to show $\theta\subseteq\abs(\theta)$. Take
$(a_{1},\dots,a_{m})\in\theta$. Then, because of \ref{A2b}, we have
$\{a_{1},a_{i}\}^{m}\subseteq \theta$ for each $i\in\{1,\dots,m\}$,
i.e., each $m$-tuple with only two components $a_{1}$ and $a_{i}$
belongs to $\theta$ and therefore also to $\abs(\theta)$. Thus each row and each column of the following
matrix $(a_{ij})_{i,j\in\um}$ belongs to $\abs(\theta)$: 
$a_{ij}:=
\begin{cases}
  a_{i}&\text{if } i=j,\\
  a_{1}&\text{otherwise,}
\end{cases}
$ with diagonal $(a_{1},\dots,a_{m})$.
Since $\abs(\theta)$ is transitive by Proposition~\ref{A2a}\ref{A2a-a}, the
diagonal $(a_{1},\dots,a_{m})$ belongs to $\abs(\theta)$. Thus
$\theta\subseteq\abs(\theta)$.

\ref{A2c-B}: By definition we have $\psi^{\updownarrow m}\in
[\psi]_{(\land,=)}$ (see \ref{A2c-2}). Since
$\psi\in\Eq(A)\subseteq\gQuord(A)$ we get $\psi^{\updownarrow
  m}\in\gQuord(A)$ by Proposition~\ref{C2}. The total symmetry of
$\psi^{\updownarrow m}$ also follows directly from the definition.

\ref{A2c-C}: $(\psi^{\updownarrow m})^{[2]}=\psi$ and
$(\theta^{[2]})^{\updownarrow m}=\theta$ directly follow from the
definitions. Thus the mappings $\psi\mapsto \psi^{\updownarrow m}$ and
$\theta\mapsto\theta^{[2]}$ are mutually inverse to each other and
obviously preserve inclusions. Thus they are lattice isomorphisms.
\end{proof}

\begin{corollary}\label{A2c-D}
Let $(A,F)$ be an algebra. Then the mappings
$\psi\mapsto \psi^{\updownarrow m}$ and $\theta\mapsto\theta^{[2]}$ are
isomorphisms between the lattices $\Con(A,F)$ and $\gCona[m](A,F)$
(where $\gCona[m](A,F):=\gEq(A)\cap \Inva[m] F$).
\end{corollary}

\begin{proof}
By definition, $\theta^{[2]}\in [\theta]_{(\land,=)}$
(see~\ref{gEq-4}) and 
$\psi^{\updownarrow m}\in[\psi]_{(\land,=)}$ as already mentioned
in the proof of \ref{A2c}\ref{A2c-B}. Thus $f\preserves\theta\implies f\preserves\theta^{[2]}$ and 
 $f\preserves\psi\implies f\preserves\psi^{{\updownarrow
     m}}$. Consequenly, in addition to \ref{A2c}\ref{A2c-C}, invariant
 relations of an algebra are mapped to 
 invariant relations under the mappings in question.
\end{proof}

\begin{fact}\label{fact:abs}
  For an $m$-ary reflexive relation $\rho$ we have 
  $\abs(\rho)=(\rho^{[2]})^{\updownarrow m}$. In fact, from the
  definition follows
  \begin{align*}
    (a_{1},\dots,a_{m})\in\abs(\rho)&\iff \forall i,j:
  \{a_{i},a_{j}\}^{m}\subseteq\rho\\
  &\iff \forall i,j:(a_{i}.a_{j})\in\rho^{[2]}\\
  &\iff (a_{1},\dots,a_{m})\in(\rho^{[2]})^{\updownarrow m}.
  \end{align*}
  
\end{fact}

With the generalized notions of reflexivity and symmetry one also can
generalize tolerances. A relation $\rho\subseteq A^{m}$ is called a
\New{generalized tolerance} if it is reflexive and totally symmetric
(cf.\ \ref{def:gEq}). We have:

\begin{proposition}\label{prop:tol}
Let $\rho$ be a generalized tolerance. Then
  the transitive closure of $\rho$ is a generalized equivalence
  relation: $\rho^{\tra}\in\gEq(A)$.
\end{proposition}

\begin{proof} The transitive closure of a reflexive relation obiously is
  reflexive. It remains to show that $\rho^{\tra}$ is
  totally symmetric. Because 
  $\rho^{\tra}=\bigcup_{n\in\N}\partial^{n}(\rho)$ (cf.\
  Section~\ref{sec:prelim}(B)), it is enough to show that $\partial(\rho)$ is
  totally symmetric whenever $\rho$ is totally symmetric.
In fact, let  $\rho\subseteq A^{m}$ and $(a_{11},\dots,a_{mm})\in\partial(\rho)$, i.e., there is
an $(m\times m)$-matrix $(a_{ij})$ with $\rho\models(a_{ij})$ (by definition of
$\partial(\rho)$, see Section~\ref{sec:prelim}(B)). 
We have to show 
$(a_{\pi1,\pi1},\dots,a_{\pi m,\pi m})\in\partial(\rho)$ 
for each permutation $\pi\in\Sym(\um)$. Since $\rho$ is totally
symmetric, one can permute rows and columns of $(a_{ij})$ arbitrarily
and rows and columns of the resulting matrix still belong to
$\rho$. In particular we have $\rho\models(a_{\pi i,\pi j})$,
consequently $(a_{\pi1,\pi1},\dots,a_{\pi m,\pi m})\in\partial(\rho)$
and we are done.
\end{proof}

\section{Factor relations and generalized partial orders}
\label{sec:gPord}

As usual we denote by $A/\psi=\{[a]_{\psi}\mid a\in A\}$ the set of
all equivalence classes 
(blocks) of an equivalence relation $\psi$ on $A$.

 \begin{definition}[factor relations]\label{A3}
For a relation $\rho\in\Rela[m](A)$ and $\psi\in\Eq(A)$, the
\New{factor relation} $\rho/\psi$ and the \New{block factor relation}
$\rho/[\psi]$ are $m$-ary relations on the set
$A/\psi$, defined as
follows:
\begin{align*}\label{A3-1}
  \rho/\psi&:=\{([a_{1}]_{\psi},\dots,[a_{m}]_{\psi})\in(A/\psi)^{m}\mid 
                 (a_{1},\dots,a_{m})\in \rho\}\tag*{\ref{A3}(1)}\\
  \label{A3-2}
    \rho/[\psi]&:=\{(B_{1},\dots,B_{m})\in(A/\psi)^{m}\mid 
        B_{1}\times\ldots\times B_{m}\subseteq\rho\}.\tag*{\ref{A3}(2)}
\end{align*}
\end{definition}

Note that $\rho/[\psi]$ might be empty in general. Obviously we have
$\rho/[\psi]\subseteq\rho/\psi$.

 \begin{definition}[exchangeability]\label{def:exchangeability} Let $a,b\in
   A$ and $\rho\in\Rela[m](A)$. We say that $a$ and $b$ are
   \New{exchangeable with respect to $\rho$}
if for any $i\in\um$ and any elements $a_{1},\dots,a_{i-1},a_{i+1},\dots,a_{m}\in
A$ we have
\begin{align*}\label{exchange-1}
  (a_{1},\dots,a_{i-1},a,a_{i+1},\dots,a_{m})\in\rho\iff
  (a_{1},\dots,a_{i-1},b,a_{i+1},\dots,a_{m})\in\rho.\tag*{\ref{def:exchangeability}(1)}
\end{align*}
The binary relation
\begin{align*}\label{exchange-2}
  &\rho^{\langle 2\rangle}:=\{(a,b)\in A^{2}\mid a \text{ and }b\text{
    are exchangable w.r.t. }\rho\}\tag*{\ref{def:exchangeability}(2)}
\end{align*}
is called the \New{exchange equivalence} of $\rho$.

  Concerning the name, it will be justified
in \ref{A3a}\ref{A3a-a} that $\rho^{\langle 2\rangle}$ is really an
equivalence relation and
that this relation has the exchange property \ref{exch1} with respect
to~$\rho$.

\begin{remarknote}
For a binary $\rho$, two elements $a,b$ are exchangeable with
respect to $\rho$ if and only if $\rho(a)=\rho(b)$ and
$\rho^{-1}(a)=\rho^{-1}(b)$ where $\rho(x):=\{y\in A\mid (x,y)\in\rho\}$.
\end{remarknote}

 \end{definition}

 \begin{lemma}\label{A3a}   
   For relations $\rho,\rho_{i}\in\Rela[m](A)$ ($i\in I$) we have
   \begin{enumerate}[label=\textup{(\alph*)}]

   \item\label{A3a-a}
    $\rho^{\langle 2\rangle}$ is an equivalence relation. It is the largest equivalence relation satisfying
    the exchange property \ref{exch1} with respect to $\rho$.
    Moreover we have
    \begin{align*}\label{A3a-a1}\tag*{\ref{A3a}(a1)}
      (a_{1},\dots,a_{m})\in\rho&\iff 
[a_{1}]_{\rho^{\langle 2\rangle}}\times\ldots
  \times[a_{m}]_{\rho^{\langle 2\rangle}}\subseteq\rho\\
   \label{A3a-a2}\tag*{\ref{A3a}(a2)}
    &\iff ([a_{1}]_{\rho^{\langle 2\rangle}},\ldots
       [a_{m}]_{\rho^{\langle 2\rangle}})\in\rho/[\rho^{\langle 2\rangle}],
    \end{align*}
 in particular, $\rho/\rho^{\langle 2\rangle}=\rho/[\rho^{\langle
   2\rangle}]$.

   \item \label{A3a-b}$\rho^{\langle 2\rangle}\subseteq \rho^{[2]}$.

\item\label{A3a-c} 
   $\bigcap_{i\in I}\rho_{i}^{\langle2\rangle}\subseteq(\bigcap_{i\in
      I}\rho_{i})^{\langle2\rangle}$.
   \end{enumerate}
 \end{lemma}
 
   \begin{proof}

\ref{A3a-a}: 
Reflexivity and symmetry directly follows from
Definition~\ref{exchange-2} and \ref{exchange-1}. Concerning transitivity, let
$(a,b),(b,c)\in\rho^{\langle 2\rangle}$. Then
\begin{align*}
  (a_{1},\dots,a_{i-1},a,a_{i+1},\dots,a_{m})\in\rho &\iff
  (a_{1},\dots,a_{i-1},b,a_{i+1},\dots,a_{m})\in\rho\\ &\iff
  (a_{1},\dots,a_{i-1},c,a_{i+1},\dots,a_{m})\in\rho,
\end{align*}
showing $(a,c)\in\rho^{\langle 2\rangle}$ according to
Definition~\ref{exchange-2}.

To show the exchange property
let $(a_{1},\dots,a_{m})\in\rho$ and
$(a_{i},b_{i})\in\rho^{\langle 2\rangle}$ for
$i\in\{1,\dots,m\}$. Then we have succesively
\begin{align*}
  (a_{1},\dots,a_{m})\in\rho&\iff(b_{1},a_{2},a_{3}\dots,a_{m})\in\rho\text{
                              (since $(a_{1},b_{1})\in\rho^{\langle
                              2\rangle}$)}\\
    &\iff (b_{1},b_{2},a_{3},\dots,a_{m})\in\rho\text{
                              (since $(a_{2},b_{2})\in\rho^{\langle
                              2\rangle}$)}\\
                  &\iff\dots\\
    &\iff(b_{1},b_{2},b_{3},\dots,b_{m})\in\rho\text{
                              (since $(a_{m},b_{m})\in\rho^{\langle
                              2\rangle}$)},
\end{align*}
what shows the exchange property. Now $\rho/\rho^{\langle
  2\rangle}=\rho/[\rho^{\langle 2\rangle}]$ directly follows from the
definitions \ref{A3-1} and \ref{A3-2}.

Concerning the quivalences \ref{A3a-a1} and \ref{A3a-a2}, note that \ref{A3a-a1} expresses the exchange property \ref{exch1}
(the implication ``$\Longleftarrow$'' is trivial since
$a_{i}\in[a_{i}]_{\rho^{\langle2\rangle}}$), and \ref{A3a-a2} follows from
Definition~\ref{A3-2}. 

It remains to show that $\rho^{\langle2\rangle}$ is the largest
equivalence with exchange property. In fact: let $\psi\in\Eq(A)$ has
the exchange property w.r.t.\ $\rho$ and take $(a,b)\in\psi$. Then
$(a_{1},\dots,a_{i-1},a,a_{i+1},\dot,a_{m})\in\rho$ if and only if 
$(a_{1},\dots,a_{i-1},b,a_{i+1},\dot,a_{m})\in\rho$ by the exchange
property $\ref{def:exchange}$ (note that $(a_{j},a_{j})\in\psi$
trivially), thus $a$ and $b$ are exchangeable, i.e.,
$(a,b)\in\rho^{\langle2\rangle}$, by Definition~\ref{exchange-1}. Thus $\psi\subseteq\rho^{\langle2\rangle}$.

\ref{A3a-b}: Let $(a,b)\in\rho^{\langle 2\rangle}$, i.e.,
$b\in[a]_{\rho^{\langle 2\rangle}}$. From \ref{A3a-a} we
conclude $\{a,b\}^{m}\subseteq[a]_{\rho^{\langle
    2\rangle}}\times\ldots\times[a]_{\rho^{\langle 2\rangle}}\subseteq
\rho$, thus $(a,b)\in\rho^{[2]}$ by Definition \ref{gEq-4}.

\ref{A3a-c}:\vspace*{-5ex}\begin{align*}
  (a,b)\in\bigcap_{i\in I}\rho_{i}^{\langle2\rangle} &\iff
\forall i\in I: (a,b)\in\rho_{i}^{\langle2\rangle}\\&\iff
\forall i\in I: a,b \text{ exchangeable with respect to }\rho_{i}\\
&\stackrel{\ref{exchange-1}}{\implies} a,b \text{ exchangeable with respect to }\bigcap_{i\in I}\rho_{i}\\
&\iff (a,b)\in (\bigcap_{i\in I}\rho_{i})^{\langle2\rangle}.&&&&\text{\qed}
\end{align*}
\renewcommand{\qedsymbol}{}
   \end{proof}

 \begin{remark}\label{rem:notmonotone} The monotonicity-property
\ref{lem:intersection1}\ref{intersection1-c} does not hold for $\rho\mapsto\rho^{\langle
     2\rangle}$, even not for generalized quasiorders. For example,
   let $A=\{0,1,2,3\}$ and $\delta_{3}:=\{(a,a,a)\mid a\in A\}$. Then
   $\sigma:=\{0,1\}^{3}\cup\delta_{3}$ and $\rho:=\sigma\cup\{(0,2,3)\}$
   are generalized quasiorders,
   $\sigma\subseteq\rho$ but $\sigma^{\langle
     2\rangle}\not\subseteq\rho^{\langle 2\rangle}$ because 
   $(0,1)\in\sigma^{\langle 2\rangle}\setminus \rho^{\langle
     2\rangle}$. In particular we have $\rho^{\langle
     2\rangle}\subsetneqq \rho^{[2]}$, since 
$(0,1)\in\rho^{[2]}\setminus \rho^{\langle 2\rangle}$. Moreover,
$\rho^{\langle2\rangle}=\Delta_{A}$.
 \end{remark}

Under some conditions factorization preserves the property of being a
generalized 
quasiorder as the implications ``$\Longrightarrow$'' in the following
proposition show:

 \begin{proposition}\label{A4}  Let $\rho\in\gQuorda[m](A)$ and 
   $\psi\in\Eq(A)$. Then we have
   \begin{enumerate}[label=\textup{(\alph*)}]
   \item\label{A4-a} $\psi\subseteq\rho^{[2]}\iff 
                   \rho/[\psi]\in\gQuord(A/\psi)$.
   \item\label{A4-b} $\psi\subseteq \rho^{\langle 2\rangle}
     \iff\rho/\psi\in\gQuord(A/\psi)$ and $\rho/\psi=\rho/[\psi]$.
   \item\label{A4-c}
     $(\rho/[\rho^{[2]}])^{[2]}=\Delta_{A/\rho^{[2]}}$,
     $\abs(\rho/\rho^{[2]})=\Delta^{(m)}_{A/\rho^{[2]}}$, 
     $(\rho/\rho^{\langle 2\rangle})^{\langle 2\rangle}=
      \Delta_{A/\rho^{\langle 2\rangle}}$.
   \end{enumerate}
 \end{proposition}

\begin{proof}

\ref{A4-a}: ``$\Longrightarrow$'': At first we show that $\rho/[\psi]$ is reflexive: Let $B=[a]_{\psi}$ be a block
of $\psi$ and let $b_{1},\dots,b_{m}\in B$
($i\in\{1,\dots,m\}$). By assumption $\psi\subseteq\rho^{[2]}$, we also have
$(b_{i},b_{j})\in\rho^{[2]}$ for all 
$i,j$. Consequently, $(b_{1},\dots,b_{m})\in(\rho^{[2]})^{\updownarrow
m}=\abs(\rho)\subseteq\rho$ (see Definition~\ref{A2c-2} and
Fact~\ref{fact:abs}). Since the $b_{i}\in B$ were chosen
arbitrarily, we conclude $B\times\ldots\times B\subseteq\rho$, i.e.,
$(B,\dots,B)\in\rho/[\psi]$ (cf.~\ref{A3-2}).

Now we show that $\rho/[\psi]$ is transitive: Let $\rho/[\psi] \models (B_{ij})$ for an
$m\times m$-matrix of blocks $B_{ij}$ of $\psi$. Choose $b_{ij}\in B_{ij}$
($i,j\in\um$). Then, by definition of a block factor relation, we have
$\rho\models(b_{ij})$. Since $\rho$ is transitive, we get
$(b_{11},\dots,b_{mm})\in\rho$. Because the $b_{ii}\in B_{ii}$ were chosen
arbitrarily, we can conclude $B_{11}\times\ldots\times
B_{mm}\subseteq\rho$, i.e., $(B_{11},\dots,B_{mm})\in \rho/[\psi]$.

``$\Longleftarrow$'': Let $(a,b)\in\psi$ and take some
$(a_{1},\dots,a_{m})\in\{a,b\}^{m}$. Then
\begin{align*}
  ([a_{1}]_{\psi},\dots,[a_{m}]_{\psi})=([a]_{\psi},\dots,[a]_{\psi})\in\rho/[\psi]
\end{align*}
since $\rho/[\psi]$ (as generalized quasiorder) is reflexive by
assumption. 
Thus, by de\-fini\-tion of
$\rho/[\psi]$, we have
$(a_{1},\dots,a_{m})\in\rho$ . Consequently $\{a,b\}^{m}\subseteq\rho$, i.e.,
$(a,b)\in\rho^{[2]}$. This shows $\psi\subseteq\rho^{[2]}$.

\ref{A4-b}: ``$\Longrightarrow$'':  Let
$\psi\subseteq\rho^{\langle2\rangle}$. 
If $\rho/\psi=\rho/[\psi]$ then from 
$\rho^{\langle 2\rangle}\subseteq\rho^{[2]}$ (cf.\ \ref{A3a}\ref{A3a-b})
and \ref{A4-a} (as just proved) follows
$\rho/\psi\in\gQuord(A)$. 
Thus it 
remains to prove 
$\rho/\psi\subseteq\rho/[\psi]$ (because
$\rho/\psi\supseteq\rho/[\psi]$ holds trivially). In fact, let
$(B_{1},\dots,B_{m})\in\rho/\psi$. Then, by definition, there exist
$a_{1}\in B_{1},\dots,a_{m}\in B_{m}$ such that
$(a_{1},\dots,a_{m})\in\rho$. From $\psi\subseteq\rho^{\langle2\rangle}$ and
\ref{A3a}\ref{A3a-a} we conclude $B_{1}\times\ldots\times
B_{m}=[a_{1}]_{\psi}\times\ldots\times[a_{m}]_{\psi}\subseteq[a_{1}]_{\rho^{\langle2\rangle}}\times\ldots\times[a_{m}]_{\rho^{\langle2\rangle}}\subseteq\rho$,
i.e., $(B_{1},\dots,B_{m})\in\rho/[\psi]$. This proves
$\rho/\psi\subseteq\rho/[\psi]$ and we are done.

``$\Longleftarrow$'': Let $(a,b)\in\psi$. We show
$(a,b)\in\rho^{\langle2\rangle}$, i.e., $a$ and $b$ are
exchangeable. In fact, if $(a_{1},\dots,a,\dots,a_{m})\in\rho$ then 
$([a_{1}]_{\psi},\dots,[a]_{\psi},\dots,[a_{m}]_{\psi})\in\rho/\psi=\rho/[\psi]$,
thus
$(a_{1},\dots,b,\dots,a_{m})\in
[a_{1}]_{\psi}\times\ldots\times[a]_{\psi}\times\ldots\times[a_{m}]_{\psi}
\subseteq\rho$
(note $b\in[a]_{\psi}$).

\ref{A4-c}:
To show $(\rho/[\rho^{[2]}])^{[2]}=\Delta_{A/\rho^{[2]}}$, let
$([a],[b])\in(\rho/[\rho^{[2]}])^{[2]}$ (for the time being, $[a]$
will denote $[a]_{\rho^{[2]}}$, we have to show $[a]=[b]$), i.e., 
$\{[a],[b]\}^{m}\subseteq \rho/[\rho^{[2]}]$ (cf.\ \ref{gEq-4}).
Therefore $[a_{1}]\times\ldots\times[a_{m}]\subseteq\rho$ for
$a_{1},\dots,a_{m}\in\{a,b\}$ by Definition~\ref{A3-2}, what implies
$(a_{1},\dots,a_{m})\in\rho$, consequently $\{a,b\}^{m}\subseteq\rho$, i.e.,
$(a,b)\in\rho^{[2]}$ and therefore $[a]=[b]$.

To show  $\abs(\rho/\rho^{[2]})=\Delta^{(m)}_{A/\rho^{[2]}}$, note
$(\abs(\rho/\rho^{[2]}))^{[2]}\subseteq
(\rho/\rho^{[2]})^{[2]}=\Delta_{A/\rho^{[2]}}$ (as just shown). Thus
$\abs(\rho/\rho^{[2]})=((\abs(\rho/\rho^{[2]}))^{[2]})^{\updownarrow
  m}=(\Delta_{A/\rho^{[2]}})^{\updownarrow m}
      =\Delta_{A/\rho^{[2]}}^{(m)}$
(the first equality follows from the fact that \ref{A2c}\ref{A2c-C}
can be applied because $\abs(\rho)\in\gEq(A)$ by \ref{A2a}\ref{A2a-a}).

Finally, to show $(\rho/\rho^{\langle 2\rangle})^{\langle 2\rangle}
        =\Delta_{A/\rho^{\langle 2\rangle}}$, 
let $([a],[b])\in(\rho/\rho^{\langle 2\rangle})^{\langle 2\rangle}$
(here and in
the following $[a]$ denotes $[a]_{\rho^{\langle 2\rangle}}$), i.e.,
$[a]$ and $[b]$ are exchangeable (cf.~\ref{def:exchangeability}) with
respect to $\rho/\rho^{\langle 2\rangle}$ and we have
\begin{align*}\tag{*}
  ([a_{1}],\dots,[a],\dots,[a_{m}])\in
         \rho/\rho^{\langle 2\rangle}\iff ([a_{1}],\dots,[b],\dots,[a_{m}])\in
         \rho/\rho^{\langle 2\rangle}
\end{align*}
(where $[a],[b]$ are at the $i$-th place, $i\in\{1,\dots,m\}$).
Thus we get
\begin{align*}
  (a_{1},\dots,a,\dots,a_{m})\in\rho&\stackrel{\ref{A3a}\ref{A3a-a}}{\iff}
     ([a_{1}],\dots,[a],\dots,[a_{m}])\in
         \rho/\rho^{\langle 2\rangle}\\
&\stackrel{\text{(*)}}{\iff} ([a_{1}],\dots,[b],\dots,[a_{m}])\in
         \rho/\rho^{\langle 2\rangle}\\
 &\stackrel{\ref{A3a}\ref{A3a-a}}{\iff}(a_{1},\dots,b,\dots,a_{m})\in\rho,
\end{align*}
i.e., $(a,b)\in\rho^{\langle 2\rangle}$ (cf.~\ref{def:exchangeability}), consequently $[a]=[b]$ and we
are done.
\end{proof}

Note that
   the above proof shows that for the implication ``$\Longleftarrow$'' in 
   \ref{A4}\ref{A4-a} only reflexivity of $\rho/[\psi]$ is
   needed. Consequently, by \ref{A4-a}, a reflexive relation
   $\rho/[\psi]$ is automatically a generalized quasiorder provided that $\rho$
   is a generalized quasiorder.

For the next Definition~\ref{def:gPord} we give here some preliminary
motivation. A binary relation $\sigma$ is antisymmetric if its
symmetric part $\sigma_{0}:=\sigma\cap\sigma^{-1}$ is 
trivial, i.e., contained in $\Delta_{A}$. 
We observe that $\sigma_{0}$ coincides with
$\tos(\sigma)=\bigcap_{\pi\in\Sym(\um)}\rho^{\pi}$
according to Definition~\ref{gEq-1}. This suggests to call
a relation $\rho\in\Rela[m](A)$ \New{antisymmetric},
if $\tos(\rho)\subseteq\Delta_{A}^{(m)}$ where
$\Delta_{A}^{(m)}:=\{(a,\dots,a)\in A^{m}\mid a\in A\}$ (what is a special
$m$-ary diagonal relation). If in addition $\rho$ is reflexiv, then $\tos(\rho)=\Delta_{A}^{(m)}$.
Therefore, an antisymmetric generalized quasiorder should be called
\New{generalized partial order} (reflexive, antisymmetric,
transitive) generalizing the binary notions.

However, for $\rho\in\gQuorda[m](A)$, we have
$\tos(\rho)=\abs(\rho)$ and $\abs(\rho)\in\gEq(A)$ by
\ref{A2a}\ref{A2a-a}, and thus
$\tos(\rho)=\abs(\rho)=\Delta_{A}^{(m)}\iff
\rho^{[2]}=_{\ref{rem:gEq}}(\abs(\rho))^{[2]}=\Delta_{A}$ what leads
to the equivalent definition in \ref{def:gPord} below.

Another nice feature of (binary) quasiorders is that $\sigma$ is
uniquely defined by its symmetric 
part $\sigma_{0}$ and the partial order on the
factor set $A/\sigma_{0}$ given by
$\sigma/\sigma_{0}=\{([a]_{\sigma_{0}},[b]_{\sigma_{0}})\in (A/\sigma_{0})^{2}\mid
(a,b)\in\sigma\}$. As we shall see in
Theorem~\ref{thm:characterization}, this result also can be
generalized, however $\rho^{[2]}$ cannot play the role of
$\sigma_{0}$ (the theorem does not hold for this in general), but one has to take
$\rho^{\langle2\rangle}$ instead. Because
$\rho^{\langle2\rangle}\subseteq\rho^{[2]}$
(cf.~\ref{A3a}\ref{A3a-b}) we get a weaker condition 
for the factor structure what led to the name \New{weak generalized
  partial order}. Clearly, each generalized partial order is also a
weak generalized order. For (binary) quasiorder these notions coincide
(since then $\sigma_{0}=\rho^{[2]}=\rho^{\langle2\rangle}$).

\begin{definition}\label{def:gPord} An $m$-ary generalized quasiorder $\rho$ is called
  \New{weak generalized partial order} or \New{generalized partial
    order}, resp., if the exchange equivalence
  $\rho^{\langle 2\rangle}$ or the binary symmetric part $\rho^{[2]}$,
  respectively, is trivial, i.e., equals $\Delta_{A}$.
The set of all weak generalized partial orders or generalized partial
orders, resp., on $A$ is denoted by $\wgPord(A)$ and $\gPord(A)$, resp. 
\end{definition}

  As obvious examples we mention that each $m$-ary diagonal relation except
  the full relation $A^{m}$ is a generalized partial order,
  i.e., $D_{A}^{(m)}\setminus\{A^{m}\}\subseteq\gPorda[m](A)$ for
  $m\in\N_{+}$, because $\{a,b\}^{m}\subseteq\rho$ (i.e., $(a,b)\in\rho^{[2]}$) implies that the
  projection to two components of $\rho$ never can be $\Delta_{A}$ if
  $a\neq b$. Clearly $A^{m}\in\gEq(A)$. Further examples
  of generalized partial orders will appear in connection with
  rectangular algebras in Section~\ref{sec:rectangular}.

The ternary relation $\rho$ defined in Remark~\ref{rem:notmonotone}
is an example for a weak
generalized partial order which is not a generalized partial order.

Note that the  set  $\gQuorda[m](A)$ is a finite lattice with respect
to inclusion 
(due to $A^{m}\in\gQuorda[m](A)$ and Lemma~\ref{C1}\ref{C1-3}).

\begin{proposition}\label{prop:gPordideal}  
  
$\gPorda[m](A)$ is an (order) ideal in the lattice $\gQuorda[m](A)$, i.e.,
$\rho\in\gPorda[m](A)$, $\sigma\in\gQuorda[m](A)$ and 
    $\sigma\subseteq\rho$ implies $\sigma\in\gPord(A)$.
    In particular, $\gPorda[m](A)$ is closed under intersections.
 
\end{proposition}

\begin{proof}
  $\sigma\subseteq\rho$ implies $\sigma^{[2]}\subseteq\rho^{[2]}$ by
  \ref{lem:intersection1}\ref{intersection1-c}. Thus
  $\rho^{[2]}=\Delta_{A}$ implies $\sigma^{[2]}=\Delta_{A}$.
\end{proof}

\begin{theorem}\label{thm:characterization}
  \begin{enumerate}[label=\textup{(\Alph*)}]
  \item \label{char-A}
Let $\rho\in\gQuorda[m](A)$. Then $\rho^{\langle2\rangle}\in\Eq(A)$,
$\rho/\rho^{\langle2\rangle}\in\wgPord(A/\rho^{\langle2\rangle})$ and we have
\begin{align*}\label{char-1}\tag*{\ref{thm:characterization}(1)}
  \rho&=\{(a_{1},\dots,a_{m})\in A^{m}\mid
  ([a_{1}]_{\rho^{\langle 2 \rangle}},\dots,[a_{m}]_{\rho^{\langle 2
  \rangle}})\in \rho/\rho^{\langle 2 \rangle}\}\\
&=\bigcup\{B_{1}\times\ldots\times B_{m}\mid (B_{1},\dots,B_{m})\in
                                                     \rho/\rho^{\langle
                                                     2 \rangle}\}.
\end{align*}
\item \label{char-B}Let $\sigma\in\Eq(A)$ and $\tau\in\wgPorda[m](A/\sigma)$. Then
\begin{align*}\label{char-2}\tag*{\ref{thm:characterization}(2)}
  \rho&:=\{(a_{1},\dots,a_{m})\in A^{m}\mid
  ([a_{1}]_{\sigma},\dots,[a_{m}]_{\sigma})\in \tau\}\\
&\phantom{:}=\bigcup\{B_{1}\times\ldots\times B_{m}\mid (B_{1},\dots,B_{m})\in
                                                     \tau\}
\end{align*}
is a generalized quasiorder, $\rho^{\langle 2\rangle}=\sigma$ and
$\rho/\rho^{\langle 2\rangle}=\tau$.

\item\label{char-C} Let $\rho$ and $\tau$ be as is in \ref{char-B}. Then $\tau\in\gPord(A)\iff \rho^{\langle2\rangle}=\rho^{[2]}$.
  \end{enumerate}
\end{theorem}

\begin{proof} \ref{char-A}: The factor relation 
$\rho/\rho^{\langle 2\rangle}$ is a generalized quasiorder by
\ref{A4}\ref{A4-b} and it is a weak partial order since
$(\rho/\rho^{\langle2\rangle})^{\langle2\rangle}=\Delta_{A/\rho^{\langle2\rangle}}$ by
\ref{A4}\ref{A4-c}. Moreover, $\rho/\rho^{\langle
  2\rangle}=\rho/[\rho^{\langle 2\rangle}]$ also by
\ref{A4}\ref{A4-b}. Therefore, \ref{char-1} is a direct consequence of
\ref{A3a-a2}.

\ref{char-B}: At first we show $\rho\in\gQuord(A)$. Reflexivity is
clear (because $\tau$ is reflexive). To show transitivity, let
$\rho\models (a_{ij})_{i,j\in\um}$. By definition of $\rho$ we have
$\tau\models ([a_{ij}]_{\sigma})_{i,j\in\um}$. 
Thus $([a_{11}]_{\sigma},\dots,[a_{mm}]_{\sigma})\in\tau$ by
transitivity of $\tau$. Consequently $(a_{11},\dots,a_{mm})\in\rho$ by
definition of $\rho$.

To see $\rho^{\langle 2\rangle}=\sigma$, observe that (by definition
\ref{char-2}) $a,b\in A$ are exchangeable with respect to $\rho$,
i.e.,
 $(a,b)\in\rho^{\langle 2\rangle}$, if and only if
$[a]_{\sigma},[b_{\sigma}]$ are exchangeable with respect
to $\tau$, i.e.,
$([a]_{\sigma},[b]_{\sigma})\in\tau^{\langle2\rangle}$. Because $\tau$ is
a weak generalized partial order, i.e.,
$\tau^{\langle2\rangle}=\Delta_{A/\sigma}$, the latter is equivalent
to
$[a]_{\sigma}=[b]_{\sigma}$, i.e., $(a,b)\in\sigma$.

It remains to prove $\tau=\rho/\rho^{\langle 2\rangle}$. In fact,
 $([a_{1}]_{\sigma},\dots,[a_{m}]_{\sigma})\in\tau \iff
(a_{1},\dots,a_{m})\in\rho\iff
([a_{1}]_{\rho^{\langle 2\rangle}},\dots,[a_{m}]_{\rho^{\langle 2\rangle}})
 \in\rho/\rho^{\langle 2\rangle}\iff 
([a_{1}]_{\sigma},\dots,[a_{m}]_{\sigma})\in\rho/\rho^{\langle 2\rangle}$ 
where the 
first equivalence follows from the definition \ref{char-2} of $\rho$, the second
and third equivalence follow from \ref{char-A} and
$\rho^{\langle 2\rangle}=\sigma$ (as just proved).

\ref{char-C}: Taking into account \ref{char-B} we have to prove 
\begin{align*}
  \tau=\rho/\rho^{\langle2\rangle}\in\gPord(A)&\iff
  \rho^{\langle2\rangle}=\rho^{[2]}, \text{ i.e., }\\
 (\rho/\rho^{\langle2\rangle})^{[2]}=\Delta_{A/\rho^{\langle2\rangle}}
    &\iff \rho^{\langle2\rangle}=\rho^{[2]}.
\end{align*}
"$\Longleftarrow$'':
$(\rho/\rho^{\langle2\rangle})^{[2]}
  =(\rho/\rho^{[2]})^{[2]}\stackrel{\ref{A4}\ref{A4-c}}{=}\Delta_{A/\rho^{[2]}}
  =\Delta_{A/\rho^{\langle2\rangle}}$.

"$\Longrightarrow$'': Assume 
$(\rho/\rho^{\langle2\rangle})^{[2]}=\Delta_{A/\rho^{\langle2\rangle}}$. We
have to show $\rho^{[2]}\subseteq\rho^{\langle2\rangle}$ (because 
$\rho^{\langle2\rangle}\subseteq\rho^{[2]}$ by \ref{A3a}\ref{A3a-b}).
Let $(a,b)\in\rho^{[2]}$, i.e., $\{a,b\}^{m}\subseteq\rho$. This implies 
$\{[a]_{\rho^{\langle2\rangle}},[b]_{\rho^{\langle2\rangle}}\}^{m}\subseteq\rho/\rho^{\langle2\rangle}$
by definition of a factor relation (cf.\ \ref{A3-1}). Consequently 
$([a]_{\rho^{\langle2\rangle}},[b]_{\rho^{\langle2\rangle}})\in
(\rho/\rho^{\langle2\rangle})^{[2]}=\Delta_{A/\rho^{\langle2\rangle}}$
(the last equation by assumption), thus
$[a]_{\rho^{\langle2\rangle}}=[b]_{\rho^{\langle2\rangle}}$, i.e.,
$(a,b)\in\rho^{\langle2\rangle}$, and we are done.
\end{proof}

As an immediate consequence of Theorem
\ref{thm:characterization}\ref{char-A} and \ref{char-B} we have:
\begin{corollary}\label{cor:characterization}
  There is a one-to-one correspondence between
  generalized quasiorders on $A$ and arbitrary pairs consisting of an
  equivalence relation $\sigma$ on $A$ and a weak generalized partial
  order on $A/\sigma$, given by
$\rho\mapsto (\rho^{\langle 2\rangle},\rho/\rho^{\langle 2\rangle})$. 
 In particular, $\rho^{\langle 2\rangle}=\rho'^{\langle 2\rangle}$
 and $\rho/\rho^{\langle 2\rangle}=\rho'/\rho'^{\langle 2\rangle}$
 implies $\rho=\rho'$ for any $\rho,\rho'\in\gQuord(A)$.\qed
\end{corollary}

In view of Theorem~\ref{thm:characterization} it arises the question
under which conditions we have equality
$\rho^{\langle2\rangle}=\rho^{[2]}$. This is relevant because then one of the two characterizing
parts of a generalized quasiorder, namely 
$\rho/\rho^{\langle2\rangle}$, is not only a weak generalized partial
order but a generalized partial order. An answer is given in the
next propositions.

\begin{proposition}\label{prop:gEq3}
  Let $\rho\in\gQuorda[m](A)$, let $\lambda$ be the canonical
  mapping $\lambda:x\mapsto[x]_{\rho^{[2]}}$ and
  $\tau:=\lambda(\rho)$. Then the following are equivalent:
  \begin{enumerate}[label=\textup{(\roman*)}]
  \item\label{gEq3-i} $\rho^{\langle2\rangle}=\rho^{[2]}$,
  \item\label{gEq3-ii} $\rho=\lambda^{-1}(\lambda(\rho))$,
  \item\label{gEq3-iii} $\rho$ is the largest $\sigma\in\gQuorda[m](A)$ with the
    property $\lambda(\sigma)=\tau$.
  \end{enumerate}

\end{proposition}

\begin{proof} Note $\ker\lambda=\rho^{[2]}$. Since
  $\rho^{\langle2\rangle}$ is the largest equivalence relation with exchange
  property (cf.\ \ref{A3a}\ref{A3a-a}) with respect to $\rho$, the equivalence 
\ref{gEq3-i}$\iff$\ref{gEq3-ii} follows from
Proposition~\ref{prop:lambdaexch}.

\ref{gEq3-ii}$\iff$\ref{gEq3-iii}: Obviously, $\rho$
is the full 
preimage of $\tau$ under $\lambda$ if and only if $\rho$ is the largest $\sigma$
with $\lambda(\sigma)=\tau$.
\end{proof}

 \begin{corollary}\label{prop:gEq2}
Let $\theta\in\gEqa[m](A)$. Then $\theta^{\langle
  2\rangle}=\theta^{[2]}$, in particular we have
$(\abs(\rho))^{\langle2\rangle}=(\abs(\rho))^{[2]}$ for $\rho\in\gQuord(A)$.
 \end{corollary}

\begin{proof}
According to Proposition \ref{prop:gEq3} it is enough to prove
$\lambda^{-1}(\lambda(\theta))=\theta$ for the canonical mapping
$\lambda:x\mapsto[x]_{\theta^{[2]}}$. The inclusion ``$\supseteq$'' is
trivial, so it remains to prove $\lambda^{-1}(\lambda(\theta))\subseteq\theta$.
Let $(a_{1},\dots,a_{m})\in\lambda^{-1}(\lambda(\theta))$, i.e.,
$(\lambda(a_{1}),\dots,\lambda(a_{m}))\in\lambda(\theta)$. Therefore
exists $(b_{1},\dots,b_{m})\in\theta$ such that 
$(\lambda(a_{1}),\dots,\lambda(a_{m}))=(\lambda(b_{1}),\dots,\lambda(b_{m}))$,
i.e., $(a_{i},b_{i})\in\ker\lambda=\theta^{[2]}$. From \ref{A2c-1} we
conclude $\forall i,j: (b_{i},b_{j})\in\theta^{[2]}$. Thus
$(a_{i},b_{i}),(b_{i},b_{j}),(a_{j},b_{j})\in\theta^{[2]}$ for all
$i,j$. Since $\theta^{[2]}$ is an equivalence relation this implies 
$\forall
a_{i},a_{j}:(a_{i},a_{j})\in\theta^{[2]}$, i.e.,
$(a_{1},\dots,a_{m})\in\theta$ by \ref{A2c-1}, and we are done.

In particular the result applies to $\abs(\rho)$ since
$\abs(\rho)\in\gEq(A)$ by \ref{A2a}\ref{A2a-a}.
\end{proof}

 \begin{remark}\label{rem:A4}
By Proposition~\ref{A2c}, generalized equivalence relations $\theta$ can be
reduced to binary equivalence relations (namely $\theta^{[2]}$). This
also holds for block factor relations: for given $\theta\in\gEq(A)$ and
$\psi\in\Eq(A)$ we have
$\theta/[\psi]=(\theta^{[2]}/[\psi])^{\updownarrow m}$.

In fact, 
\begin{align*}
  ([a_{1}]_{\psi},\dots,[a_{m}]_{\psi})\in\theta/[\psi]
&\stackrel{\ref{A3-2}}{\iff} [a_{1}]_{\psi}\times\ldots\times[a_{m}]_{\psi}\subseteq\theta\stackrel{\ref{A2c}\ref{A2c-C}}{=}(\theta^{[2]})^{\updownarrow m}\\
&\stackrel{\ref{A2c-2}}{\iff} \forall i,j: [a_{i}]_{\psi}\times[a_{j}]_{\psi}\subseteq\theta^{[2]}\\
&\stackrel{\ref{A3-2}}{\iff} \forall i,j: ([a_{i}]_{\psi},[a_{j}]_{\psi})\in\theta^{[2]}/[\psi]\\
&\stackrel{\ref{A2c-2}}{\iff} ([a_{1}]_{\psi},\dots,[a_{m}]_{\psi})\in(\theta^{[2]}/[\psi])^{\updownarrow m}.
\end{align*}
 \end{remark}

\section{The generalized quasiorders of maximal
  clones given by 
equivalence relations or partial orders}\label{sec:maxclones} 

Let $\rho\in\Rela[2](A)$ be a non-trivial
equivalence relation or a partial order (reflexive,
antisymmetric and transitive relation on $A$) with least and greatest element
(denoted by $0$ and $1$, respectively). It is well-known that then the
clone $\Pol\rho$ is a maximal clone in the clone lattice over the base
set $A$.

We ask for all invariant relations
of these maximal clones which are generalized quasiorders, or from the
relational point of view, we ask for a characterization of the set
$[\rho]_{(\exists,\land,=)}\cap \gQuord(A)
    =\Inv\Pol\rho\cap\gQuord(A)=\gQuord\Pol\rho$.

In particular cases, treated first (see Theorems~\ref{thm:eqrel},
\ref{thm:latticeorder} and \ref{thm:boolean}), the answer is nice:
\emph{every} invariant 
relation is a generalized quasiorder: $\gQuord\Pol\rho=\Inv\Pol\rho$. Moreover, in these cases we have
$\Inv\Pol\rho=[\rho]_{(\exists,\land,=)}=[\rho]_{(\land,=)}$ what
allows an explicit description of the generalized quasiorders as given in
the following lemma.

\begin{lemma}\label{lem:conjunctions}
Let $\rho\in\Rela[2](A)$ and $\sigma\in[\rho]_{(\land,=)}$ be $m$-ary. Then
$\sigma$ has the form
\begin{align*}
  \sigma=\{(a_{1},\dots,a_{m})\in A^{m}\mid  (a_{i},a_{j})\in\rho&\text{ for all } (i,j)\in
  E \text{ and }\\
  a_{i'}=a_{j'}&\text{ for all }(i',j')\in E'\},
\end{align*}
  for some subsets $E$ and $E'$ of $\{1,\dots,m\}^{2}$.
\end{lemma}
The proof is obvious: with $E$ and $E'$, respectively, one collects all atomic
formulas of the form $(x_{i},x_{j})\in\rho$ and $x_{i'}=x_{j'}$,
respectively (other atomic formulas do not exist).\qed

Depending on properties (e.g. transitivity) of the concrete $\rho$ this
representation can be further simplified.

\begin{theorem}\label{thm:eqrel} Let $\rho\in\Equ(A)$ and $F=\Pol\rho$.
 Then each
  invariant relation of the clone
  $F$ is a generalized quasiorder and can be obtained from
  $\rho$ by a pp-formula without quantifiers, i.e.,
  \begin{align*}
    \gQuord F=\Inv F=[\rho]_{(\exists,\land,=)}=[\rho]_{(\land,=)}.
  \end{align*}

  \end{theorem}

  \begin{proof}
   We have $[\rho]_{(\land,=)}\subseteq_{\ref{C2}}\gQuord \Pol\rho\subseteq\Inv
   \Pol\rho=[\rho]_{(\exists,\land,=)}$. Thus it is enough to prove
   $[\rho]_{(\exists,\land,=)}\subseteq [\rho]_{(\land,=)}$ to make
   all inclusions to be the equality what will
   finish the proof of the proposition.

Let $\sigma_{\phi}\in[\rho]_{(\exists,\land,=)}$ be a relation
determined by some pp-formula $\phi(x_{1},\dots,x_{m})\equiv\exists
y_{1},\dots,y_{s}:\Phi(x_{1},\dots,x_{m},y_{1},\dots,y_{s})$ with free
variables $x_{1},\dots,x_{m}$ and bounded variables
$y_{1},\dots,y_{s}$ where $\Phi$ is a conjunction of atomic formulas
of the form $(z,z')\in\rho$ or $z=z'$ for variables $z,z'$. 
If there
is some $(z,z')\in \rho$ with $z\in\{x_{1},\dots,x_{m}\}$ and
$z'\in\{y_{1},\dots,y_{s}\}$, then the bounded variable $z'$ can be
substituted by $z$ everywhere and we get a formula $\phi'$ with less
bounded variables but defining the same relation
$\sigma_{\phi'}=\sigma_{\phi}$
 (this follows from the
reflexivity, symmetry and transitivity of $\rho$, for instance,
 $[\exists z': (z,z')\in\rho\land(x_{i},z')\in\rho]\iff[(z,x_{i})\in\rho]$,
here ``$\Longrightarrow$'' follows from transitivity (and
symmetry), for ``$\Longleftarrow$'' take $z':=z$).

If this reduction is done as long as possible and there remain
bounded variables then they are not ``connected'' with any free
variable and can be deleted (because they always can be evaluated by an arbitrary (fixed) constant). Thus
there is a quantifier-free formula $\phi''$ defining
$\sigma_{\phi}=\sigma_{\phi''}$. 
  \end{proof}

Remark: If $\rho$ is trivial, i.e., if $\rho\in\{\Delta_{A},\nabla_{A}\}$, then
$F=\Op(A)$ and $\Inv F$ is the set of all trivial (i.e.,
diagonal, cf.~Section~\ref{sec:prelim}(A)) relations. Otherwise $F$ is a maximal clone as already mentioned.

Now we turn to partial orders $\rho$ with $0$ and $1$ (then
$\Pol\rho$ is a maximal clone). An interesting observation is that then
each nontrivial generalized quasiorder $\sigma\in\gQuord\Pol\rho$ is
already a generalized 
partial order (i.e., $\sigma^{[2]}=\Delta_{A}$, cf.\
\ref{def:gPord}), but we postpone this result to the publication
\cite{JakPR2026} (in preparation).

At first we focus on lattice orders $\rho$. It
is remarkable that lattices can be characterized by the
 ``nice'' property, mentioned above (before \ref{lem:conjunctions}), as the following theorem shows.

\begin{theorem}\label{thm:latticeorder}
  Let $(A,\rho)$ be a poset with least and greatest element.The following are
  equivalent:
  \begin{enumerate}[label=\textup{(\Alph*)}]
\item\label{latticeorder-A} $\rho$ is a lattice order,
\item\label{latticeorder-B} $\gQuord \Pol\rho
    =\Inv\Pol\rho$,
\item\label{latticeorder-C}  $[\rho]_{(\land,=)}=[\rho]_{(\exists,\land,=)}$\,.
  \end{enumerate}
\end{theorem}

\begin{proof}
 \ref{latticeorder-C}$\implies$\ref{latticeorder-B} directly follows from the following
  inclusion-chain (note that $\rho$ is a (generalized) quasiorder,
  thus the first inclusion follows from Proposition~\ref{C2}):
$[\rho]_{(\land,=)}\subseteq
[\rho]_{(\exists,\land,=)}\cap \gQuord(A)=\gQuord
\Pol\rho\subseteq\Inv\Pol\rho=[\rho]_{(\exists,\land,=)}$.

\ref{latticeorder-B}$\implies$\ref{latticeorder-A} (indirect):
Assume $\rho$ is \textbf{not} a lattice order. 
Then $\gQuord\Pol\rho\neq\Inv\Pol\rho $ because of Example~\ref{ex:1}.

\ref{latticeorder-A}$\implies$\ref{latticeorder-C}: 
We can apply the \textsc{Baker/Pixley}-theorem (\cite[2.1(1)-(2)]{BakP75}):
\it 
if (and only if) an algebra $\mathbf{A}$ 
has a $(d+1)$-ary near unanimity term operation then each subalgebra of a
(finite) direct product of algebras in $\mathcal{V}(\mathbf{A})$ (the
variety generated by $\mathbf{A}$) is uniquely
determined by it $d$-ary projections. \rm

If $\rho$ is a lattice order, then the algebra $\mathbf{A}=(A,\Pol\rho)$ has a
ternary near unanimity operation, namely the majority operation $(x\land
y)\lor(y\land z)\lor(z\land x)$. Consequently,
each invariant relation $\sigma\in\Inva[m]\Pol\rho$ (which is nothing
else than a subalgebra of the
direct power $(A,\Pol\rho)^{m}\in\mathcal{V}(\mathbf{A})$) is uniquely determined by its binary
projections $\sigma_{ij}:=\pr_{ij}(\sigma)=\{(x_{i},x_{j})\in A^{2}\mid \exists
x_{k}: (x_{1},\dots,x_{m})\in\sigma\}$.

Since
$\Pol\rho$ is a maximal clone and therefore $[\rho]_{(\exists,\land,=)}$ is a
minimal relational clone, for each non-trivial $\sigma_{ij}$ we must
have $\rho\in[\sigma_{ij}]_{(\exists,\land,=)}$. This implies
$\sigma_{ij}\in\{\rho,\rho^{-1}\}$ (what can be proved directly,
however an explicite proof can be found in the proof of
\cite[4.3.7]{PoeK79}).

Thus we have $\sigma_{ij}\in\{\Delta_{A},\nabla_{A},\rho,\rho^{-1}\}$,
therefore $\sigma_{ij}$ is definable by an atomic formula (existential
free), namely
$x_{i}=x_{j}$ or $x_{i}=x_{i}$ or $(x_{i},x_{j})\in\rho$ or
$(x_{j},x_{i})\in\rho$, respectively.

Consider the relation $\sigma':=\bigwedge_{1\leq i,j\leq
  m}\sigma_{ij}(x_{i},x_{j})$. Then, by construction, $\sigma'$ belongs to
$[\rho]_{(\land,=)}$ and has the same binary projections as $\sigma$.
Applying the \textsc{Baker/Pixley}-theorem, we get $\sigma=\sigma'$. 
Consequently $[\rho]_{(\exists,\land,=)}=[\rho]_{(\land,=)}$.
\end{proof}

Non-lattice
orders do not satisfy the conditions in
Theorem~\ref{thm:latticeorder}. We demonstrate this explicitly  
with the following example (provided by \textsc{G. Gyenizse}).

\begin{example}\label{ex:1} 
Let $(A,\rho)$ be a poset with least element $0$ and
  greatest element $1$ which is
  not a lattice. 
Then there exist $a,b\in A$
with no least upper bound, i.e., there exist $c,d$, both covering $a$
and $b$ (equivalently, there exist $c,d$ with no greatest lower
bound), see Figure~\ref{fig:1}.

\vspace{-3ex}
\begin{figure}[h]
  \begin{center}
      \includegraphics{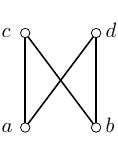}
  \end{center}\vspace*{-5ex}
\caption{Subgraph existing for the diagram of a non-lattice order relation}\label{fig:1}
\end{figure}

We also use the notation $\leq$ for $\rho$ here.
Let 
 \begin{align*}
    \sigma:=\{(x_{1},x_{2},x_{3},x_{4})\in A^{4}\mid \exists y:
x_{1}\leq y\leq x_{3}, x_{2}\leq y\leq x_{4}\}.
  \end{align*}
Then $\sigma$ is reflexive and belongs to $[\rho]_{(\exists,\land,=)}$ by construction, but $\sigma$ is \textbf{not} a
generalized quasiorder. In fact, we have
$\sigma\models\left(
  \begin{smallmatrix}
    a&0&c&d\\ 0&b&c&d\\ c&c&c&1\\d&d&1&d
  \end{smallmatrix}
\right)$ (it is easy to check that each row and column belongs to
$\sigma$), but the diagonal 
$(a,b,c,d)$ of this matrix does not belong 
to $\sigma$. Therefore $\sigma$ is not transitive and thus not a
generalized quasiorder. Consequently, $\gQuord\Pol\rho\subsetneqq\Inv\Pol\rho=[\rho]_{(\exists,\land,=)}$.
\end{example}

We now consider the special case $A=\{0,1\}$.
Let $\Mo$ be the maximal clone of
monotone Boolean operations, i.e., operations preserving the binary relation
$\rho:=\leq=\{(0,0),(0,1),(1,1)\}\subseteq A^{2}$:
$\Mo=\Pol\rho$. Since $\rho$ is a lattice order, by
Theorem~\ref{thm:latticeorder} we know $[\rho]_{(\land,=)}=\Inv\Mo=\gQuord\Mo\subseteq
\gQuord(A)$. But for the two-element set $A$ we get even more, namely
equality for the last inclusion.

  \begin{theorem}\label{thm:boolean} The generalized quasiorders on $A=\{0,1\}$
    coincide with the 
  invariant relations of the maximal clone $\Mo$ of monotone Boolean
  functions:
  \begin{align*}
    \gQuord(A)=\Inv\Mo =[\rho]_{(\land,=)}.
  \end{align*}
\end{theorem}

\begin{proof}  As mentioned above we have
  $\Inv\Mo\subseteq\gQuord(A)$. It
  remains to show $\gQuord(A)\subseteq\Inv\Mo$. Let $\sigma\in\gQuord(A)$. We recall some
  notation and facts from \cite[Section~2, 2.2(***), Theorem~3.8]{JakPR2024}: For $M:=\End\sigma$ we
  have $\Pol \sigma=M^{*}:=\{f\in\Op(A)\mid \trl{f}\subseteq M\}$
  ($\trl{f}$ denotes the set of (unary) translations of $f$). In
  particular we have $(\End\rho)^{*}=\Pol\rho=\Mo$ since $\rho$ is a
  quasiorder and therefore also a generalized quasiorder.

Because
  of reflexivity of $\sigma$, the endomorphism monoid $M$ must contain
  all constants. On $\{0,1\}$ there are only two such monoids
  $M_{0}:=\{c_{0},c_{1},\id_{A}\}=\End\rho$  and
  $M_{1}:=\{c_{0},c_{1},\id_{A},\neg\}=A^{A}$ (where obviously
  $M_{1}^{*}=\Op(A)$).

Case 1: $\End\sigma=M_{0}$. Then
$\Pol\sigma=(\End\sigma)^{*}=M_{0}^{*}=(\End\rho)^{*}=\Pol\rho$ what implies
$\sigma\in\Inv\Pol\sigma=\Inv\Pol\rho = \Inv\Mo$.

Case 2:  $\End\sigma=M_{1}$. Then $\Pol\sigma=M_{1}^{*}=\Op(A)$,
thus $\sigma\in\Inv\Pol\sigma=\Inv\Op(A)\subseteq \Inv\Mo$.

In each case $\sigma\in\Inv\Mo$ what finishes the proof. 
\end{proof}

Motivated by the condition $[\rho]_{(\exists,\land,=)}=[\rho]_{(\land,=)}$ in
Theorems~\ref{thm:eqrel} and \ref{thm:latticeorder} and the fact
$[\rho]_{(\land,=)}\subseteq\gQuord(A)$ (cf.~\ref{C2}), 
there arises the challenging conjecture that
$[\rho]_{(\exists,\land,=)}\cap\gQuord(A)=[\rho]_{(\land,=)}$ might hold for
arbitrary generalized quasiorders $\rho\in\gQuord(A)$. We are more
modest here and formulate it only for partial orders (due to Theorem~\ref{thm:latticeorder} it is sufficient to
consider only non-lattice orders):

\begin{conjecture}\label{conj:partialorder}
  Let $\rho$ be a partial order with $0$ and $1$. Then 
  \begin{align*}
    \gQuord\Pol\rho&=[\rho]_{(\land,=)} \text{ or, equivalently,}\\
    [\rho]_{(\exists,\land,=)}\cap\gQuord(A)&=[\rho]_{(\land,=)}.
  \end{align*}
\end{conjecture}

For a non-lattice order $\rho$ we already know
$[\rho]_{(\land,=)}\subseteq[\rho]_{(\exists,\land,=)}\cap\gQuord(A)$ (by~\ref{C2}),
however it is not clear if there exists a ``counter\-example, i.e., a
generalized quasiorder which
is definable by a pp-formula but not by a quantifier free
pp-formula. In \cite{JakPR2026} we shall deal with
Conjecture~\ref{conj:partialorder} in more detail and give some
partial results.

\section{Generalized quasiorders in rectangular
  algebras}\label{sec:rectangular}

In Proposition~\ref{thm:rectangular} we shall see that rectangular
algebras are a source for many generalized quasiorders, in particular
generalized partial orders, and that one
axiom (namely (\textbf{AB}${}^{i}_{f}$)) characterizes the property
for the graph 
$\graph{f}$ of a function $f$ to be a generalized
quasiorder.

\begin{definition}\label{R1}  An algebra $(A,(f_{i})_{i\in I})=(A,F)$
 (of finite type) is called \New{rec\-tan\-gular algebra} if for all fundamental
  operations $f,g\in F$ ($f$ $n$-ary, $g$ $m$-ary) the following
  identities are satisfied:
\begin{enumerate}[leftmargin=8ex]
  \item[$(\mathbf{ID}_{f})$] $f(x,x, \ldots,x) \approx x$ \hfill(idempotence)
  \item[$(\mathbf{AB}^{i}_{f})$]
   $f(x_1, \ldots, x_{i-1},f(y_1,
     \ldots, y_{i-1},x_i,y_{i+1},\ldots ,
    y_n),x_{i+1}, \ldots, x_n) \approx f(x_1,...,x_n)$\\
     \hspace*{\fill} (absorption in each place $i\in\{1,\dots,n\}$)

  \item[$(\mathbf{C}_{f,g})$]
    $f(g(x_{11},\dots,x_{1m}),\dots,g(x_{n1},\dots,x_{nm}))$\\ 
  \hspace*{3ex}$\approx g(f(x_{11},\dots,x_{n1}),\dots,f(x_{1m},\dots,x_{nm}))$
  \hfill (commuting operations)
  \end{enumerate}

If $f$ is idempotent, then the absorption identities together
are equivalent to the following single identity
\begin{enumerate}[leftmargin=8ex]
\item[$(\mathbf{AB}_{f})$]
 $f(f(x_{11}, \ldots, x_{1n}),f(x_{21}, \ldots ,x_{2n}),
     \ldots, f(x_{n1}, \ldots,x_{nn})) \approx
     f(x_{11}, \ldots, x_{nn})$.
\end{enumerate}


\end{definition}

\begin{remark}\label{R2}
 Rectangular algebras with a single binary operation are just the
 rectangular bands. Rectangular algebras with a single $n$-ary
 operation are P\l onka's diagonal algebras \cite{Plo1966}.
 Moreover, algebras satisfying $(\mathbf{C}_{f,g})$ for all
 fundamental operations $f,g$ are called \New{entropic}. Idempotent entropic
 algebras were investigated by \textsc{Romanowska} and
 \textsc{Smith}, see, e.g., \cite{RomS1989}, called \New{modes} in
 \cite{RomS2002}. Thus rectangular algebras are special modes.

  It is known that the variety of rectangular algebras (of fixed
  finite type)
  is generated by all projection algebras (of the same type), i.e.,
  algebras where each fundamental operation is a projection. In
  particular, for finite type (i.e., $|I|$ finite), each finite rectangular
  algebra is isomorphic to a finite direct product of projection
  algebras, and each finite projection algebra is a 
  subalgebra of a finite direct product of two-element projection
  algebras. Moreover, the variety of rectangular algebras (as well as
  the variety of modes, cf.~\cite{RomS2002}) of fixed type is a
  so-called \New{solid variety}, i.e., the identities in Definition~\ref{R1}
  hold not only for the fundamental operations but also for all term
  operations. For details we refer to \cite{PoeR93}.
\end{remark}

The following proposition provides a large number of (higher-ary) generalized
quasiorders (which are even generalized partial orders), all being
graphs of operations. Moreover, those operations whose graphs are
generalized quasiorders, can be characterized (under certain
assumptions) by the property \ref{R1}$(\mathbf{AB}_{f})$.

\begin{theorem}\label{thm:rectangular}
  \begin{enumerate}[label=\textup{(\roman*)}]
  \item \label{rect1}
  Let $f: A^{n}\to A$ satisfy {\rm\ref{R1}$(\mathbf{C}_{f,f})$}. Then $f$ satisfies $(\mathbf{AB}_{f})$
    if and only if the graph $\graph{f}$ of $f$,
  \begin{align*}
    \graph{f}:=\{(a_{1},\dots,a_{n},b)\in A^{n+1}\mid f(a_{1},\dots,a_{n})=b\},
  \end{align*}
 is an $(n+1)$-ary generalized quasiorder. 

  \item\label{rect2} The graph
 $\graph{t}$ of each term operation $t$ of a rectangular algebra $(A,F)$ is
 a generalized partial order.
 \end{enumerate}
\end{theorem}

\begin{remarknote}
  \ref{rect1} can be formulated as
  follows: For an 
 entropic algebra $(A,f)$ (cf.~\ref{R2}), the graph $\graph{f}$ is a generalized
 quasiorder if and only if 
 $f$ satisfies $(\mathbf{AB}_{f})$.
\end{remarknote}

\begin{proof} \ref{rect1}: Let $\graph{f}\models M$ for a matrix
    $M=
    \left( \begin{smallmatrix}
      a_{11}&\dots& a_{1n}& b_{1}\\
      \vdots& &\vdots&\vdots\\
      a_{n1}&\dots& a_{nn}& b_{n}\\
      c_{1}&\dots& c_{n}&d
    \end{smallmatrix}\right)$,
in particular (considering the first $n$ rows and columns only), we have $f(a_{i1},\dots,a_{in})=b_{i}$ and
  $f(a_{1i},\dots,a_{ni})=c_{i}$ for $i\in\{1,\dots,n\}$. Then, by
  ({\bf C}${}_{f,f}$), we automatically also have the condition for
  the last column and row:
$f(b_{1},\dots,b_{n})=d=f(c_{1},\dots,c_{n})$. 
Therefore, in $M$ the $a_{ij}$ can be chosen arbitrarily. Now it is
clear that the diagonal of $M$
belongs to $\graph{f}$, i.e., $f(a_{11},\dots,a_{nn})=d$, if and only
if $f$ satisfies ({\bf AB}${}_{f}$).

\ref{rect2}: Since in a rectangular algebra each term operation $t$ satisfies the
identities ({\bf ID}${}_{t}$), {\rm({\bf C}${}_{t,t}$)} and {\rm({\bf
    AB}${}_{t}$)} (by solidity as mentioned in Remark~\ref{R2}), we
conclude from \ref{rect1} that $\graph{t}$ is a generalized
quasiorder. Because $\{a,b\}^{n+1}\in \graph{t}$ implies
$(a,\dots,a,a),(a,\dots,a,b)\in\graph{t}$ and therefore
$a=t(a,\dots,a)=b$, we have $(\graph{t})^{[2]}=\Delta_{A}$, i.e.,
$\graph{t}$ is a generalized partial order (cf.~\ref{def:gPord} and \ref{gEq-4}).
\end{proof}

\begin{example}\label{ex:rect} The most well-known examples of
  rectangular algebras are rectangular bands. A \New{rectangular band}
  is usually defined as  a semigroup $(A,*)$ satisfying
\begin{align*}
  x*x&\approx x\tag{idempotence}\\
  x*y*z&\approx x*z\tag{absorption}
\end{align*}

The identities $(\mathbf{ID}_{*})$, $(\mathbf{AB}_{*})$ and
$(\mathbf{C}_{*,*})$ from Definition~\ref{R1} easily can be
checked. From Proposition~\ref{thm:rectangular}\ref{rect1} we conclude that
then the graph
  \begin{align*}
    \{(a_{1},a_{2},b)\in A^{3}\mid a_{1}*a_{2}=b\}
  \end{align*}
of $*$  is a ternary generalized partial order.
The standard example of a rectangular band is $(A\times A,*)$ (for some
nonempty set $A$) with the multiplication $*$ defined by $(a,b)*(c,d):=(a,d)$
(what can be visualized by drawing rectangles on lattice points in the plane if $A=\N$).
\end{example}

\begin{remark}\label{rem:abelian} For operations $f$ as considered in
  Theorem~\ref{thm:rectangular}\ref{rect1} exists an interesting connection to
  strongly abelian algebras (in the sense of tame congruence theory,
  \cite[Definition 3.10]{HobM88}, first defined in \cite{McK1983}). It was proved in \cite[Proposition~4.7]{Scholle2010} that
  an algebra $(A,f)$ with a single fundamental operation $f$ satisfies
  {\rm({\bf ID}${}_{f}$)} and {\rm({\bf  AB}${}_{f}$)} if and only if
  it is strongly abelian. 
\end{remark}

\section{Further research}\label{sec:further}

Because the investigation of generalized quasiorders just started with
their invention in \cite{JakPR2024} the field for further research is
wide and it remains to be seen which direction is most promising. In
Section~6 of \cite{JakPR2024} there are already mentioned some
directions ($\gQuord$-completeness as generalization of affine
completeness , characterization of u-closed
monoids and of the Galois closures $\gQuord\End Q$, structure of the lattices $\gQuorda[m](A)$ as well as of the
lattice $\cK_{A}^{(m)}$ of all such lattices for fixed base set $A$). 

With Section~\ref{sec:maxclones} in mind we may ask for the
generalized quasiorders of maximal clones (according to I. Rosenberg's
classification). This will be done in a publication \cite{JakPR2026}
which is still in preparation: each generalized quasiorder for a
maximal clone except those from Section~\ref{sec:maxclones} must be
trivial (i.e., a diagonal relation). The only open problem is the challenging
Conjecture~\ref{conj:partialorder}, for which at least some partial results
also will be given in \cite{JakPR2026}.

General quasiorders can serve for the investigation of the lattice
$\cL_{A}$ of all clones on a finite set $A$. For instance, for each
set $Q$ of generalized quasiorders of the form $Q=\gQuord F_{0}$ (for
some clone $F_{0}$) the set 
$\{F\leq \Op(A)\mid \gQuord F\subseteq Q\}$ 
forms an order-filter in the clone lattice.
For $Q=D_{A}$
we get the filter $\{F\leq \Op(A)\mid \gQuord F =D_{A}\}$ where the maximal elements are the maximal clones
mentioned above. What are the minimal elements in this filter?

In general one may ask how generalized quasiorders can influence the
(algebraic) structure of an algebra. For example, due to
Remark~\ref{rem:abelian} and Theorem~\ref{thm:rectangular}, 
an idempotent algebra $(A,f)$ with self-commuting $f$ (i.e., an
idempotent entropic algebra) is a strongly
abelian algebra if and only if  $\graph{f}$ is a generalized quasiorder.

\subsection*{Acknowledgement} 
The research of the first author was supported by the Slovak VEGA grant
1/0152/22. The authors thank David Stanovsk\'y for the Remark~\ref{rem:abelian}.

\def\cprime{$'$} \def\cprime{$'$}

Danica Jakub{\'\i}kov\'a-Studenovsk\'a:~Institute of
  Mathematics, P.J.~\v{S}af\'arik University,~Ko\v sice, \texttt{studend@gmail.com},

Reinhard P{\"o}schel: Institute of Algebra, Technische Universität Dresden,\\ \texttt{reinhard.poeschel@tu-dresden.de},

S\'andor Radeleczki: Institute of Mathematics, University of
    Miskolc,\\ \texttt{sandor.radeleczki@uni-miskolc.hu}.

\end{document}